\documentclass[11pt,twoside,a4paper]{article}

\usepackage{amssymb,amsmath,amsthm,amsfonts,mathtools,mathrsfs,esint}
\usepackage{times}
\usepackage{enumerate,enumitem}
\usepackage{cite,titletoc}
\usepackage{graphicx,float}
\usepackage{booktabs,tabularx,array}
\usepackage{aliascnt}
\usepackage{microtype}
\usepackage[colorlinks=true,linkcolor=blue,citecolor=red,urlcolor=cyan]{hyperref}
\usepackage[nameinlink,capitalize,noabbrev]{cleveref}

\allowdisplaybreaks
\numberwithin{equation}{section}
\setlist{itemsep=2pt,topsep=4pt,parsep=0pt,partopsep=0pt}

\newcommand{\R}{\mathbb{R}}

\newcommand{\cA}{\mathcal{A}}

\newcommand{\BMO}{\mathrm{BMO}}

\newcommand{\avgint}{\mathop{\fint}}
\newcommand{\norm}[1]{\lVert #1\rVert}

\newcommand{\charac}{\mathbf{1}}

\newcommand{\twA}{\mathsf{A}}

\newcommand{\pairing}[2]{\langle #1,#2\rangle}

\theoremstyle{plain}
\newtheorem{theorem}{Theorem}[section]

\newaliascnt{proposition}{theorem}
\newtheorem{proposition}[proposition]{Proposition}
\aliascntresetthe{proposition}

\newaliascnt{lemma}{theorem}
\newtheorem{lemma}[lemma]{Lemma}
\aliascntresetthe{lemma}

\newaliascnt{corollary}{theorem}
\newtheorem{corollary}[corollary]{Corollary}
\aliascntresetthe{corollary}

\theoremstyle{definition}
\newaliascnt{definition}{theorem}
\newtheorem{definition}[definition]{Definition}
\aliascntresetthe{definition}

\newaliascnt{example}{theorem}
\newtheorem{example}[example]{Example}
\aliascntresetthe{example}

\newaliascnt{question}{theorem}
\newtheorem{question}[question]{Question}
\aliascntresetthe{question}

\theoremstyle{remark}
\newaliascnt{remark}{theorem}
\newtheorem{remark}[remark]{Remark}
\aliascntresetthe{remark}

\crefname{theorem}{Theorem}{Theorems}
\crefname{proposition}{Proposition}{Propositions}
\crefname{lemma}{Lemma}{Lemmas}
\crefname{corollary}{Corollary}{Corollaries}
\crefname{definition}{Definition}{Definitions}
\crefname{example}{Example}{Examples}
\crefname{question}{Question}{Questions}
\crefname{remark}{Remark}{Remarks}

\title{Properties and applications of Lorentz--Muckenhoupt classes}

\author{Dinghuai Wang\footnote{Wang Dinghuai (\texttt{Wangdh1990@126.com}) is supported by NSFC (No.~12101010). } \quad and \quad Qian Zhang
    \vspace{0.5cm}\\
\small
School of Mathematics and Statistics, Anhui Normal University, Wuhu 241002, China.\\
}
\vspace{0.5cm}
\date{}

\begin{document}

\maketitle
\thispagestyle{empty}

\begin{abstract}
In this paper, through the introduction of Lorentz--Muckenhoupt classes, we systematically investigate the boundedness of maximal operators on
multiplier weighted Lorentz spaces.
As applications, we give the characterization of the commutators of fractional integrals, which yields a partial answer to an open question proposed by D. Cruz-Uribe.
Second, Hardy inequalities in Lorentz spaces are established, with the critical case $p=d$, in which the classical Hardy inequality fails.
Finally, we apply the Lorentz estimates to fractional Schr\"odinger equations with singular potentials.
\vskip 0.2 true cm


\vskip 0.2 true cm

\end{abstract}


\section{Introduction}
For $1<p<\infty$, the Muckenhoupt classes $A_p$ is defined by
$$[W]_{A_{p}}:=\sup_{Q}\avgint_Q W(y)dy \Big(\avgint_Q W(y)^{1-p'}dy\Big)^{p-1}<\infty.$$
Muckenhoupt's theorem gives a complete characterization of the weights for which the Hardy--Littlewood maximal operator
 $$
 M f(x):=\sup_{Q\ni x}\avgint_Q|f(y)|\,dy
 $$
is bounded on weighted Lebesgue spaces:
$$
M: L^p(W)\to L^p(W)
\quad\text{if and only if}\quad
W\in A_p,
\qquad 1<p<\infty.
$$
Meanwhile, for
\begin{equation}\label{eq:SC}
 0\le\alpha<d,\qquad 1<p<\infty,
 \qquad \frac1q=\frac1p-\frac\alpha d>0,
\end{equation}
the $A_{p,q}$ weight condition characterizes the one-weight boundedness of the fractional maximal operator $M_\alpha$ from $L^p(w^p)$ to $L^q(w^q)$.
Here
$$[w]_{A_{p,q}}:=\sup_{Q}\Big(\avgint_Q w(y)^qdy\Big)^{1/q} \Big(\avgint_Q w(y)^{-p'}dy\Big)^{1/p'}<\infty.$$
and
 $$
 M_\alpha f(x)=\sup_{Q\ni x}|Q|^{\alpha/d}\frac1{|Q|}\int_Q|f(y)|\,dy.
 $$
The Muckenhoupt classes $A_p$ and $A_{p,q}$ play a fundamental role in harmonic analysis. They satisfy factorization
properties, reverse H\"older inequalities, extrapolation theory,
and sharp operator estimates; see~\cite{Muckenhoupt1972,Jones1980,
HarboureMaciasSegovia1988,CruzUribeMartellPerez2011}. We refer to \cite{MuckenhouptWheeden1971,MuckenhouptWheeden1974,
LaceyMoenPerezTorres2010} for the classical and sharp quantitative theory.
Sparse domination and positive dyadic methods provide a modern framework for
these estimates; see
\cite{Lerner2013,CondeAlonsoRey2016,HanninenHytoneenLi2016,
DiPlinioHytoneenLi2020,Li2017TwoWeight,Li2018Sparse,LiMoenSun2015}.

\medskip

There are two different ways to introduce a weight into a Lorentz space.
The traditional measure-weighted space $L^{p,r}(W\,dx)$ uses the distribution
function
 $$
 \lambda\longmapsto W(\{x:|f(x)|>\lambda\})
 $$
and rearranges $f$ with respect to the measure $W(x)\,dx$.  The maximal
operator in this setting was studied by Chung, Hunt, and Kurtz
\cite{ChungHuntKurtz1982}; further weighted Lorentz and restricted weak-type
results appear in
\cite{Sawyer1990,AccomazzoEtAl2023,KokilashviliMeskhi2020}.

The present paper considers a different setting, in which the weight acts as a multiplier.
\begin{equation*}
 L_w^{p,r}:=\{f:fw\in L^{p,r}(\mathbb R^d)\},
 \qquad
 \|f\|_{L_w^{p,r}}:=\|fw\|_{L^{p,r}(dx)}.
\end{equation*}
If $r=p$, one has
$L_{W^{1/p}}^{p,p}=L^p(W\,dx)$. For $r\ne p$, the spaces
$L^{p,r}_{W^{1/p}}$ and $L^{p,r}(W\,dx)$ are generally different; see \cref{ex:TCE}.
Multiplier weak-type estimates go back to Muckenhoupt and Wheeden and have
recently reappeared in connection with matrix weights, sparse operators, and
quantitative weak bounds; see \cite{Sweeting2024}.  Their endpoint classes
contain critical multipliers outside the classical $A_p$ and $A_{p,q}$
classes.  This motivates the study of multiplier classes defined by
normalized Lorentz norms.

\medskip

We now introduce the Lorentz--Muckenhoupt classes studied in this paper and
summarize the corresponding maximal operator results.  For
$1<p\le q<\infty$ and $1\le r,s\le\infty$, define
\begin{equation*}
 [w]_{A_{(p,r),(q,s)}}
 :=\sup_Q
 \|w\|_{L^{q,s}(Q)}
 \|w^{-1}\|_{L^{p',r'}(Q)},
\end{equation*}
where
 $$
 \|h\|_{L^{a,b}(Q)}:=|Q|^{-1/a}\|h\mathbf1_Q\|_{L^{a,b}}.
 $$
The principal purpose of the paper is to study $A_{(p,r),(q,s)}$.  We prove
duality, monotonicity and a complete classification
of power weights.  The two-weight formulation in
\cref{thm:TWT,cor:TWT} is used only to compute
the exact norm of the fractional averaging operators and to derive the
necessary condition appearing in the table below.  In particular,
in the case $r\ge p$ and $s\le q$, the class coincides with $A_{p,q}$;
see \cref{thm:AA}.

The class $A_{(p,r),(q,s)}$ is always necessary for the boundedness of maximal operator, but it is not sufficient for every pair $(r,s)$.
The following table summarizes
the results proved in this paper.  For $p>1$, the indices satisfy
\eqref{eq:SC}.  When $p=1$, write
$q=d/(d-\alpha)$ for $\alpha>0$ and $q=1$ for $\alpha=0$.

\begin{center}
\footnotesize
\renewcommand{\arraystretch}{1.30}
\begin{tabularx}{\textwidth}{>{\raggedright\arraybackslash}p{0.095\textwidth}
 >{\raggedright\arraybackslash}p{0.19\textwidth}
 >{\raggedright\arraybackslash}X
 >{\raggedright\arraybackslash}p{0.19\textwidth}}
\toprule
principal range & secondary indices &
$M_\alpha:L_w^{p,r}\to L_w^{q,s}$ & conclusion\\
\midrule
$p>1$ & $r>s$ & the estimate fails even for $w\equiv1$ &
\cref{ex:r-greater-s}\\
$p>1$ & $p\le r\le s\le q$ & if and only if $w\in A_{(p,r),(q,s)}$
, equivalently $w\in A_{p,q}$ &
\cref{thm:SSM}\\
$p>1$ & $r=p$, $q\le s\le\infty$ & if and only if
$w\in A_{(p,p),(q,s)}$ &
\cref{thm:COC,thm:MQO}\\
$p>1$ & $r=1$, $s<\infty$ & $w\in A_{(p,1),(q,s)}$ is necessary but not
sufficient & \cref{ex:WIF}\\
$p>1$ & $r=1$, $s=\infty$ & $w\in A_{(p,1),(q,\infty)}$ is necessary;
$w\in A_{(p,p),(q,\infty)}$ is sufficient; a complete characterization remains open &
\cref{rem:PRS}\\
$p>1$ & $1<r<p$, $r\le s$ & $w\in A_{(p,r),(q,s)}$ is necessary; $w\in A_{(p,p),(q,s)}$ is sufficient;
a complete characterization remains open & \cref{rem:II}\\
$p>1$ & $p<r\le s$, $s>q$ & $w\in A_{(p,r),(q,s)}$ is necessary;
a complete characterization remains open
 & \cref{rem:LILO}\\
$p=1$ & $r=1$, $s<\infty$ & the estimate fails even for $w\equiv1$ &
\cref{prop:P1O}\\
$p=1$ & $r=1$, $s=\infty$ & $w\in A_{(1,1),(q,\infty)}$ is necessary;
$w\in A_{(1,1),(q,q)}$ is sufficient;
a complete characterization remains open (Muckenhoupt--Wheeden problem)&
\cref{prop:P1E}\\
$p=1$ & $r>1$ & the estimate fails even for $w\equiv1$ &
\cref{ex:P1LI}\\
\bottomrule
\end{tabularx}
\end{center}

Beyond the ranges $p\leq r\leq s\leq q$ and $r=p$ with $q\leq s\leq\infty$, our main positive result is the following characterization: \begin{equation}\label{eq:MM}
M_\alpha:L_w^{p,r}\longrightarrow L_w^{q,s} \quad\Longleftrightarrow\quad w\in A_{(p,p),(q,s)}.
\end{equation}
When $\alpha=0$, \eqref{eq:MM} yields
 $$
 M:L_w^{p,p}\longrightarrow L_w^{p,s}
 \quad\Longleftrightarrow\quad
 w\in A_{(p,p),(p,s)},
 \qquad p\le s\le\infty.
 $$

\medskip

We next describe the applications that motivate the preceding theory.

\medskip
\noindent\textbf{1. Partial answer to an open question of D.~Cruz-Uribe.}
For a locally integrable function $b$, let
 $$
 [b,I_\alpha]f:=bI_\alpha f-I_\alpha(bf)
 $$
be the commutator of $I_\alpha$, where
$$I_{\alpha}(f)(x):=\int_{\mathbb{R}^d}\frac{f(y)}{|x-y|^{d-\alpha}}dy.$$
Classical commutator characterizations of $\BMO$ were established in
\cite{CoifmanRochbergWeiss1976,Janson1978,Chanillo1982}.  In the two-weight
setting, however, a single pair of weights naturally leads to a
Bloom-type weighted $\BMO$ space; see
\cite{Bloom1985,HolmesRahmSpencer2016,HolmesLaceyWick2017}.  After discussing
sufficient two-weight conditions for fractional commutators, D.~Cruz-Uribe
posed the following question at the end of the paper \cite{CruzUribe2017}, and
pointed out that ``nothing is known about this question but it merits further investigation".

\begin{question}\label{ques:cruz-uribe}
Can anything be said about $b$ if there exists a pair of weights $(u, \sigma)$ (or
perhaps a family of such pairs) such that $[b,I_{\alpha}](\cdot \sigma): L^{p}(\sigma) \rightarrow L^{q}(u)$?
\end{question}

When $q$ is determined by
$$
\frac{1}{q}
=
\frac{1}{p}-\frac{\alpha}{d}
=:\frac{1}{p_{\alpha}},
$$
the result is relatively straightforward; see
\cite[Corollary 2.4]{CR2018}. The main difficulty is the case
$
p\leq q<p_{\alpha}.
$
Without loss of generality, it suffices to consider the case $q=p$.

Our first application gives a partial answer to \cref{ques:cruz-uribe}.
We prove that if
 $$
 \norm{[b,I_\alpha]f}_{L^p((wv)^p)}
 \le C_0\norm{v}_{L^{t,\infty}}
 \norm{f}_{L^p(w^p)}
 $$
for every nonnegative $v\in L^{t,\infty}$ and some $w\in A_{p,p_\alpha}$, then $b\in\BMO$.  The proof uses Lorentz duality and
 $$
 [b,I_\alpha]:
 L_w^{p,p}\longrightarrow L_w^{p_\alpha,p},
 $$
for $w\in A_{p,p_{\alpha}}$, see \cref{thm:one-weight-BMO-characterization}.

\medskip
\noindent\textbf{2. The Hardy inequality at the critical exponent.}
The classical inequality
 $$
 \Big\|\frac{u}{|x|}\Big\|_{L^p}
 \lesssim\|\nabla u\|_{L^p},\qquad 1<p<d,
 $$
fails at $p=d$.  \Cref{thm:critical-hardy-lorentz} gives the endpoint
replacement
\begin{equation*}
 \big\||x|^{-d/q}u\big\|_{L^{q,\infty}}
 \lesssim\|\nabla u\|_{L^{d,1}},
 \qquad d\le q<\infty.
\end{equation*}
In particular,
 $$
 \Big\|\frac{u}{|x|}\Big\|_{L^{d,\infty}}
 \lesssim\|\nabla u\|_{L^{d,1}}.
 $$
The estimate fails with
$L^{d,r}$ on the right for every $r>1$.  Related critical Hardy inequalities
are discussed in \cite{RuzhanskySuragan2017}.

\medskip
\noindent\textbf{3. Fractional Schr\"odinger equations.}
The weighted Hardy--Lorentz estimates lead to a singular equation in which
the two multipliers have distinct analytic roles.  Let
 $$
 \delta=\alpha-\beta-\gamma>0,
 \qquad t=\frac d\delta,
 \qquad A(x)=|x|^{\beta+\gamma}a(x).
 $$
If $A\in L^{t,\infty}$ and its norm is sufficiently small, then
\cref{thm:hardy-schrodinger} gives a unique mild solution of
 $$
 (-\Delta)^{\alpha/2}u-a(x)u=f
 $$
in
$X_\beta^{q,\rho}=\{u:|x|^{-\beta}u\in L^{q,\rho}\}$, together with
\begin{equation*}
 \||x|^{-\beta}u\|_{L^{q,\rho}}
 \le
 \frac{C}{1-C\|A\|_{L^{t,\infty}}}
 \||x|^\gamma f\|_{L^{p,\rho}}.
\end{equation*}
The critical Hardy potential $a(x)=\lambda|x|^{-\alpha}$ is included for
sufficiently small $|\lambda|$.  Related fractional Schr\"odinger operators
with Hardy-type potentials are studied in
\cite{FrankLiebSeiringer2008,Bieganowski2018}.

The structure of the paper is as follows.  Section~2
contains the Lorentz-space preliminaries and compares the two
weighted setting.  Section~3 develops the Lorentz--Muckenhoupt classes and
their structural properties.  Section~4
contains the maximal operator theorems and counterexamples.  Section~5 gives the commutator theorem, the
critical Hardy inequalities, and the fractional Schr\"odinger equationn.
\section{Preliminaries}

\subsection{Lorentz spaces and weighted Lorentz spaces}

For a measurable function $f$ on $\R^d$, let $f^*$ denote its decreasing
rearrangement.  We follow the standard Lorentz-space in
\cite{Lorentz1950,Hunt1966,BerghLofstrom1976,BennettSharpley1988}.  If $0<a<\infty$ and $0<b<\infty$, we use the normalized
Lorentz quasi-norm
\begin{equation*}
 \norm{f}_{L^{a,b}}
 :=\Big(\frac ba\int_0^\infty
 \big[t^{1/a}f^*(t)\big]^b\,\frac{dt}{t}\Big)^{1/b},
\end{equation*}
and
\begin{equation*}
 \norm{f}_{L^{a,\infty}}
 :=\sup_{t>0}t^{1/a}f^*(t).
\end{equation*}

\medskip

We use the following standard facts; see Bennett and Sharpley
\cite{BennettSharpley1988}.

\begin{lemma}\label{lem:L}
Let $1<a<\infty$.
\begin{enumerate}[label=\textup{(\roman*)}]
\item If $0<b_0\le b_1\le\infty$, then
 $$
 L^{a,b_0}\hookrightarrow L^{a,b_1}.
 $$
\item If $1\le b\le\infty$, then
\begin{equation*}
 \int_{\R^d}|fg|
 \le C_{a,b}\norm{f}_{L^{a,b}}
 \norm{g}_{L^{a',b'}}.
\end{equation*}
\item If $E$ has finite measure, $0<a_0<a_1\le\infty$, and
$0<b_0,b_1\le\infty$, then
\begin{equation*}
 |E|^{-1/a_0}\norm{f\charac_E}_{L^{a_0,b_0}}
 \lesssim
 |E|^{-1/a_1}\norm{f\charac_E}_{L^{a_1,b_1}}.
\end{equation*}
\item Let $1<a_0,a_1<\infty$, $1\le b_0,b_1,b\le\infty$, and assume
 $$
 \frac1a=\frac1{a_0}+\frac1{a_1}<1,
 \qquad
 \frac1b\le\frac1{b_0}+\frac1{b_1}.
 $$
Then O'Neil's product inequality gives
\begin{equation*}
 \norm{fg}_{L^{a,b}}
 \lesssim
 \norm{f}_{L^{a_0,b_0}}\norm{g}_{L^{a_1,b_1}}.
\end{equation*}
\end{enumerate}
\end{lemma}

For $1<a<\infty$ and $1\le b\le\infty$, the associate space of
$L^{a,b}$ is $L^{a',b'}$, with equivalence of norms.  In particular,
\begin{equation*}
 \norm{h}_{L^{a',b'}}
 \approx
 \sup_{\norm{g}_{L^{a,b}}\le1}\int_{\R^d}|gh|.
\end{equation*}
If $b=1$ or $b=\infty$, this is understood in the K\"othe-dual sense.

\vspace{0.3cm}

We next distinguish the two weighted Lorentz constructions that occur in the literature. Let $1<a<\infty$ and $1\le b\le\infty$.  Multiplier and factorization
questions for Lorentz and rearrangement-invariant spaces are discussed in
\cite{Reisner1981,KolwiczLesnikMaligranda2014}.

The following example shows that the space $L^{a,b}(\omega)$ is different from $ L_{w^{1/a}}^{a,b}$, when $a\neq b$.

\begin{example}\label{ex:TCE}
For $\Omega=(0,e^{-1})$, take $a=2$, $b=1$, and set
 $$
 \omega_1(x)=\frac{1}{x\log^2(e/x)}.
 $$
Since $\omega_1(\Omega)<\infty$, the constant function $1$ belongs to both
$L^{2,1}(\omega_1)$ and $L^2(\omega_1)$.  On the other hand,
 $$
 \omega_1(x)^{1/2}
 =x^{-1/2}\log^{-1}(e/x),
 $$
and therefore
 $$
\norm{\omega_1^{1/2}}_{L^{2,1}}
 \gtrsim\int_0^{e^{-1}}\frac{dt}{t\log(e/t)}=\infty.
 $$
Hence $L^{2,1}(\omega_1)$ is not contained in
$L^{2,1}_{\omega_1^{1/2}}$.

Conversely, on $\Omega=(0,1)$ let $a=2$, $b=\infty$, and
$\omega_2(x)=x^{-1}$.  Then
 $$
 \norm{1}_{L_{\omega_2^{1/2}}^{2,\infty}}
 =\norm{x^{-1/2}}_{L^{2,\infty}(0,1)}<\infty.
 $$
However, $\omega_2(\Omega)=\infty$ and
 $
 1\notin L^{2,\infty}(\omega_2).
 $
\end{example}

\vspace{0.3cm}

The associate space of the multiplier-weighted
Lorentz space is
\begin{equation*}
 (L_w^{a,b})'=L_{w^{-1}}^{a',b'}
\end{equation*}
with equivalence of norms.

We also use the real interpolation identity
\begin{equation}\label{eq:WI}
 (L_w^{a_0},L_w^{a_1})_{\theta,b}=L_w^{a,b},
 \qquad
 \frac1a=\frac{1-\theta}{a_0}+\frac\theta{a_1},
\end{equation}
for $1\le a_0<a_1\le\infty$, $0<\theta<1$, and
$1\le b\le\infty$.  This follows from the unweighted identity.

\subsection{BMO and off-diagonal extrapolation}

For a locally integrable function $b$, write
$\langle b\rangle_Q=|Q|^{-1}\int_Qb$.  The space $\BMO(\R^d)$ consists of
all $b$ such that
\begin{equation*}
 \norm{b}_{\BMO}
 :=\sup_Q\avgint_Q|b-\langle b\rangle_Q|<\infty.
\end{equation*}
We refer to John and Nirenberg \cite{JohnNirenberg1961} and Fefferman and
Stein \cite{FeffermanStein1972} for the classical theory.

We use classical off-diagonal extrapolation in the following form; see
Harboure, Mac\'ias, and Segovia \cite{HarboureMaciasSegovia1988} and
Cruz-Uribe, Martell, and P\'erez \cite{CruzUribeMartellPerez2011}.

\begin{theorem}\label{thm:COF}
Let $1<p_0\le q_0<\infty$ and set
$\delta_0=1/p_0-1/q_0\ge0$.  Suppose that $T$ is sublinear and that there is
an increasing function $\Psi$ such that
\begin{equation}\label{eq:base-weighted-bound}
 \norm{Tf}_{L_w^{q_0}}
 \le \Psi([w]_{A_{p_0,q_0}})\norm{f}_{L_w^{p_0}}
\end{equation}
for every $w\in A_{p_0,q_0}$ and every bounded compactly supported $f$.
Then, for every $1<p\le q<\infty$ satisfying
$1/p-1/q=\delta_0$ and every $w\in A_{p,q}$,
\begin{equation*}
 \norm{Tf}_{L_w^q}\le C\norm{f}_{L_w^p},
\end{equation*}
where $C$ depends on the indices, $[w]_{A_{p,q}}$, and $\Psi$.
\end{theorem}

\section{Lorentz--Muckenhoupt classes}

We first consider a two-weight class because it computes the norm of fractional average operator and yields the necessary $\twA^\alpha_{(p,r),(q,s)}$ condition used later.

\begin{definition}
Let $0\le\alpha<d$, $1<p\le q<\infty$, and
$1\le r,s\le\infty$.  A pair of weights $(u,\sigma)$ belongs to
$\twA^\alpha_{(p,r),(q,s)}$ if
\begin{equation*}
 [(u,\sigma)]_{\twA^\alpha_{(p,r),(q,s)}}
 :=\sup_Q |Q|^{\frac\alpha d+\frac1q-\frac1p}
 \norm{u^{1/q}}_{L^{q,s}(Q)}
 \norm{\sigma^{1/p'}}_{L^{p',r'}(Q)}<\infty.
\end{equation*}
\end{definition}
We also obtain
\begin{equation*}
 [(u,\sigma)]_{\twA^\alpha_{(p,r),(q,s)}}
 =\sup_Q |Q|^{\alpha/d-1}
 \norm{u^{1/q}\charac_Q}_{L^{q,s}}
 \norm{\sigma^{1/p'}\charac_Q}_{L^{p',r'}}.
\end{equation*}

\medskip

For a cube $Q$, define
\begin{equation*}
 \cA_{Q,\sigma}^\alpha f(x)
 :=|Q|^{\alpha/d}\Big(\avgint_Qf\sigma\Big)\charac_Q(x).
\end{equation*}

\begin{theorem}\label{thm:TWT}
Let $0\le\alpha<d$, $1<p\le q<\infty$, $1\le r,s\le\infty$, and let
$(u,\sigma)$ be a pair of weights. Then
\begin{equation*}
 \sup_Q
 \norm{\cA_{Q,\sigma}^\alpha}_{L_{\sigma^{1/p}}^{p,r}\to L_{u^{1/q}}^{q,s}}
 \approx
 [(u,\sigma)]_{\twA^\alpha_{(p,r),(q,s)}}.
\end{equation*}
The constants depend only on $p,q,r$, and $s$.
\end{theorem}

\begin{proof}
Set
 $$
 \mathfrak T:=\sup_Q
 \norm{\cA_{Q,\sigma}^\alpha}_{L_{\sigma^{1/p}}^{p,r}\to L_{u^{1/q}}^{q,s}}.
 $$
We first prove $\mathfrak T\lesssim [(u,\sigma)]_{\twA^\alpha_{(p,r),(q,s)}}$. Fix a cube $Q$ and
$f\in L_{\sigma^{1/p}}^{p,r}$. Since
 $$
 f\sigma=(f\sigma^{1/p})\sigma^{1/p'},
 $$
then
\begin{equation*}
 \Big|\int_Qf\sigma\Big|
 \lesssim
 \norm{f\sigma^{1/p}\charac_Q}_{L^{p,r}}
 \norm{\sigma^{1/p'}\charac_Q}_{L^{p',r'}}.
\end{equation*}
Therefore
\begin{align*}
 \norm{\cA_{Q,\sigma}^\alpha f}_{L_{u^{1/q}}^{q,s}}
 &=|Q|^{\alpha/d-1}
 \Big|\int_Qf\sigma\Big|
 \norm{u^{1/q}\charac_Q}_{L^{q,s}} \notag\\
 &\lesssim |Q|^{\alpha/d-1}
 \norm{u^{1/q}\charac_Q}_{L^{q,s}}
 \norm{\sigma^{1/p'}\charac_Q}_{L^{p',r'}}
 \norm{f}_{L_{\sigma^{1/p}}^{p,r}} \notag\\
 &\le [(u,\sigma)]_{\twA^\alpha_{(p,r),(q,s)}}\norm{f}_{L_{\sigma^{1/p}}^{p,r}}.
\end{align*}
Taking the supremum over $Q$ yields $\mathfrak T\lesssim [(u,\sigma)]_{\twA^\alpha_{(p,r),(q,s)}}$.

For the converse, fix $Q$ and $0<\varepsilon<1$. There exists a nonnegative function $h_{Q,\varepsilon}$ supported
in $Q$ such that
\begin{equation*}
 \norm{h_{Q,\varepsilon}}_{L^{p,r}}\le1,
 \qquad
 \int_Qh_{Q,\varepsilon}\sigma^{1/p'}
 \ge c_{p,r}(1-\varepsilon)
 \norm{\sigma^{1/p'}\charac_Q}_{L^{p',r'}}.
\end{equation*}
Define
 $$
 f_{Q,\varepsilon}:=h_{Q,\varepsilon}\sigma^{-1/p}.
 $$
Then
\begin{equation}\label{eq:TWT-1}
 \norm{f_{Q,\varepsilon}}_{L_{\sigma^{1/p}}^{p,r}}
 =\norm{h_{Q,\varepsilon}}_{L^{p,r}}\le1,
\end{equation}
and
\begin{align*}
 \norm{\cA_{Q,\sigma}^\alpha f_{Q,\varepsilon}}_{L_{u^{1/q}}^{q,s}}
 &=|Q|^{\alpha/d-1}
 \Big(\int_Qh_{Q,\varepsilon}\sigma^{1/p'}\Big)
 \norm{u^{1/q}\charac_Q}_{L^{q,s}} \notag\\
 &\gtrsim (1-\varepsilon)|Q|^{\alpha/d-1}
 \norm{u^{1/q}\charac_Q}_{L^{q,s}}
 \norm{\sigma^{1/p'}\charac_Q}_{L^{p',r'}}.
\end{align*}
By \eqref{eq:TWT-1}, the left-hand side is at most
$\mathfrak T$. Letting $\varepsilon\downarrow0$ and then taking the
supremum over $Q$ gives $[(u,\sigma)]_{\twA^\alpha_{(p,r),(q,s)}} \lesssim\mathfrak T$.
\end{proof}

\begin{corollary}\label{cor:TWT}
Let $0\le\alpha<d$, $1<p\le q<\infty$, $1\le r,s\le\infty$, and let
$(u,\sigma)$ be a pair of weights. If
\begin{equation*}
 M_\alpha(\cdot\,\sigma):
 L_{\sigma^{1/p}}^{p,r}\longrightarrow L_{u^{1/q}}^{q,s}
\end{equation*}
is bounded, then $(u,\sigma)\in\twA^\alpha_{(p,r),(q,s)}$. If
$0<\alpha<d$, the same conclusion follows from the boundedness of
\begin{equation*}
 I_\alpha(\cdot\,\sigma):
 L_{\sigma^{1/p}}^{p,r}\longrightarrow L_{u^{1/q}}^{q,s}.
\end{equation*}
\end{corollary}

\begin{proof}
Let $f\ge0$ and let $Q$ be a cube. For $x\in Q$,
\begin{equation*}
 \cA_{Q,\sigma}^\alpha f(x)
 =|Q|^{\alpha/d}\Big(\avgint_Qf\sigma\Big)\charac_Q(x)
 \le M_\alpha(f\sigma)(x).
\end{equation*}
Consequently,
 $$
 [(u,\sigma)]_{\twA^\alpha_{(p,r),(q,s)}}\lesssim \norm{\cA_{Q,\sigma}^\alpha}_{L_{\sigma^{1/p}}^{p,r}\to L_{u^{1/q}}^{q,s}}
 \le
 \norm{M_\alpha(\cdot\,\sigma)}_{L_{\sigma^{1/p}}^{p,r}\to L_{u^{1/q}}^{q,s}}.
 $$

For the fractional integral, if $x,y\in Q$, then
$|x-y|\le C_d|Q|^{1/d}$. Hence, for $x\in Q$,
\begin{align*}
 I_\alpha(f\sigma\charac_Q)(x)
 &=\int_Q\frac{f(y)\sigma(y)}{|x-y|^{d-\alpha}}\,dy \notag\\
 &\ge c_{d,\alpha}|Q|^{\alpha/d-1}\int_Qf(y)\sigma(y)\,dy
 =c_{d,\alpha}\cA_{Q,\sigma}^\alpha f(x).
\end{align*}
Using \cref{thm:TWT}, we complete the proof.
\end{proof}

\subsection{One-weight classes}

We now introduce the one--weight class, which will be used throughout the paper.

\begin{definition}
Let $1<p\le q<\infty$ and $1\le r,s\le\infty$. A weight $w$ belongs to
$A_{(p,r),(q,s)}$ if
\begin{equation*}
 [w]_{A_{(p,r),(q,s)}}
 :=\sup_Q
 \norm{w}_{L^{q,s}(Q)}
 \norm{w^{-1}}_{L^{p',r'}(Q)}<\infty.
\end{equation*}
\end{definition}

If $\alpha/d=1/p-1/q$, then
\begin{equation*}
 [w]_{A_{(p,r),(q,s)}}
 =[(w^q,w^{-p'})]_{\twA^\alpha_{(p,r),(q,s)}}.
\end{equation*}
For a cube $Q$, write
\begin{equation*}
 \cA_Q^\alpha f(x)
 :=|Q|^{\alpha/d}\Big(\avgint_Q f\Big)\charac_Q(x).
\end{equation*}
Taking $(u,\sigma)=(w^q,w^{-p'})$ in \cref{thm:TWT} and
using the change of variables $g=fw^{p'}$ gives the following one-weight
statement.

\begin{theorem}\label{thm:OWT}
Let $1<p\le q<\infty$, $1\le r,s\le\infty$, and let
$\alpha/d=1/p-1/q$. Then
\begin{equation*}
 \sup_Q\norm{\cA_Q^\alpha}_{L_w^{p,r}\to L_w^{q,s}}
 \approx [w]_{A_{(p,r),(q,s)}}.
\end{equation*}
\end{theorem}

We first establish the elementary properties of $A_{(p,r),(q,s)}$.

\begin{proposition}
Let $1<p\le q<\infty$ and $1\le r,s\le\infty$. One has
\begin{equation*}
 [w^{-1}]_{A_{(q',s'),(p',r')}}
 =[w]_{A_{(p,r),(q,s)}}.
\end{equation*}
\end{proposition}

\begin{proof}
By the definition of $A_{(p,r),(q,s)}$, we have
\begin{align*}
 [w^{-1}]_{A_{(q',s'),(p',r')}}
 &=\sup_Q |Q|^{\frac1{q'}-\frac1{p'}-1}
 \norm{w^{-1}\charac_Q}_{L^{p',r'}}
 \norm{w\charac_Q}_{L^{q,s}}.
\end{align*}
Since
 $$
 \frac1{q'}-\frac1{p'}-1
 =\frac1p-\frac1q-1,
 $$
the right-hand side equals $[w]_{A_{(p,r),(q,s)}}$.
\end{proof}

\medskip

By Lemma \ref{lem:L}, the classes $A_{(p,r),(q,s)}$ are monotone in the secondary indices.

\begin{proposition}
Let $1<p\le q<\infty$ and $1\le r,s\le\infty$.
\begin{enumerate}[label=\textup{(\roman*)}]
\item If $1\le s_0\le s_1\le\infty$, then
\begin{equation*}
 A_{(p,r),(q,s_0)}\subset A_{(p,r),(q,s_1)}.
\end{equation*}
\item If $1\le r_0\le r_1\le\infty$, then
\begin{equation*}
 A_{(p,r_1),(q,s)}\subset A_{(p,r_0),(q,s)}.
\end{equation*}
\end{enumerate}
\end{proposition}

\medskip

We can also obtain the following interpolation results for the Lorentz--Munckenhoupt weights.

\begin{theorem}
For $i\in\{0,1\}$, let
 $$
 w_i\in A_{(p_i,r_i),(q_i,s_i)},
 \qquad 1<p_i,q_i<\infty,
 \quad 1\le r_i,s_i\le\infty.
 $$
Fix $0<\theta<1$ and define
\begin{align*}
 \frac1{q_\theta}&=\frac{1-\theta}{q_0}+\frac\theta{q_1},
 &\frac1{s_\theta}&=\frac{1-\theta}{s_0}+\frac\theta{s_1},\\
 \frac1{p_\theta'}&=\frac{1-\theta}{p_0'}+\frac\theta{p_1'},
 &\frac1{r_\theta'}&=\frac{1-\theta}{r_0'}+\frac\theta{r_1'}.
\end{align*}
Then $w_\theta=w_0^{1-\theta}w_1^\theta$ satisfies
\begin{equation*}
 [w_\theta]_{A_{(p_\theta,r_\theta),(q_\theta,s_\theta)}}
 \lesssim
 [w_0]_{A_{(p_0,r_0),(q_0,s_0)}}^{1-\theta}
 [w_1]_{A_{(p_1,r_1),(q_1,s_1)}}^\theta.
\end{equation*}
\end{theorem}

\begin{proof}
Fix a cube $Q$. The definition of Lorentz norms gives
\begin{align*}
 \norm{w_0^{1-\theta}}_{L^{q_0/(1-\theta),s_0/(1-\theta)}(Q)}
 &=\norm{w_0}_{L^{q_0,s_0}(Q)}^{1-\theta},\\
 \norm{w_1^\theta}_{L^{q_1/\theta,s_1/\theta}(Q)}
 &=\norm{w_1}_{L^{q_1,s_1}(Q)}^\theta.
\end{align*}
Hence
\begin{equation}\label{eq:GOF}
 \norm{w_\theta}_{L^{q_\theta,s_\theta}(Q)}
 \lesssim
 \norm{w_0}_{L^{q_0,s_0}(Q)}^{1-\theta}
 \norm{w_1}_{L^{q_1,s_1}(Q)}^\theta.
\end{equation}
Applying the same argument to
$w_\theta^{-1}=w_0^{\theta-1}w_1^{-\theta}$ gives
\begin{equation}\label{eq:GIF}
 \norm{w_\theta^{-1}}_{L^{p_\theta',r_\theta'}(Q)}
 \lesssim
 \norm{w_0^{-1}}_{L^{p_0',r_0'}(Q)}^{1-\theta}
 \norm{w_1^{-1}}_{L^{p_1',r_1'}(Q)}^\theta.
\end{equation}
Multiplying \eqref{eq:GOF} and
\eqref{eq:GIF}, the desired result is obtained.
\end{proof}

We next compare the new class with the classical class $A_{p,q}$.

\begin{lemma}\label{lem:RHL}
Let $1<a<\infty$, $1\le b\le\infty$, and $\varepsilon>0$. Then
\begin{equation*}
 \norm{f}_{L^{a,b}(Q)}
 \lesssim_{a,b,\varepsilon}
 \Big(\avgint_Q|f|^{a(1+\varepsilon)}\Big)^{1/[a(1+\varepsilon)]}.
\end{equation*}
\end{lemma}

\begin{proof}
Since
$L^{a(1+\varepsilon),a(1+\varepsilon)}=L^{a(1+\varepsilon)}$, we get
\begin{align*}
& \norm{f}_{L^{a,b}(Q)}
=|Q|^{-1/a}\norm{f\charac_Q}_{L^{a,b}}\\
 &\lesssim |Q|^{-1/[a(1+\varepsilon)]}
 \norm{f\charac_Q}_{L^{a(1+\varepsilon)}}=\Big(\avgint_Q|f|^{a(1+\varepsilon)}\Big)^{1/[a(1+\varepsilon)]}.
\end{align*}
\end{proof}

\begin{theorem}\label{thm:AA}
Let $1<p\le q<\infty$ and $1\le r,s\le\infty$.
\begin{enumerate}[label=\textup{(\roman*)}]
\item One has
 $$
 A_{p,q}\subset A_{(p,r),(q,s)}.
 $$
\item If $s\le q$ and $r\ge p$, then
\begin{equation*}
 A_{(p,r),(q,s)}=A_{p,q}.
\end{equation*}
\end{enumerate}
\end{theorem}

\begin{proof}
If $w\in A_{p,q}$, then $w^q$ and $w^{-p'}$ belong to
$A_\infty$. Hence there exist $\varepsilon_0,\varepsilon_1>0$ such that,
uniformly over cubes $Q$,
\begin{align}
 \Big(\avgint_Qw^{q(1+\varepsilon_0)}\Big)^{1/[q(1+\varepsilon_0)]}
 &\lesssim \Big(\avgint_Qw^q\Big)^{1/q},
 \label{eq:RH-wq}\\
 \Big(\avgint_Qw^{-p'(1+\varepsilon_1)}\Big)^{1/[p'(1+\varepsilon_1)]}
 &\lesssim \Big(\avgint_Qw^{-p'}\Big)^{1/p'}.
 \label{eq:RH-wp}
\end{align}
By \cref{lem:RHL},
\begin{align*}
 \norm{w}_{L^{q,s}(Q)}
 &\lesssim
 \Big(\avgint_Qw^{q(1+\varepsilon_0)}\Big)^{1/[q(1+\varepsilon_0)]},\\
 \norm{w^{-1}}_{L^{p',r'}(Q)}
 &\lesssim
 \Big(\avgint_Qw^{-p'(1+\varepsilon_1)}\Big)^{1/[p'(1+\varepsilon_1)]}.
\end{align*}
Combining these inequalities with \eqref{eq:RH-wq}--\eqref{eq:RH-wp}
gives
 $$
 \norm{w}_{L^{q,s}(Q)}
 \norm{w^{-1}}_{L^{p',r'}(Q)}
 \lesssim [w]_{A_{p,q}}.
 $$
Taking the supremum over $Q$ proves (i).

Now assume $s\le q$ and $r\ge p$. Then $r'\le p'$. The Lorentz embeddings
$L^{q,s}\hookrightarrow L^{q,q}$ and
$L^{p',r'}\hookrightarrow L^{p',p'}$ imply
\begin{align*}
 \Big(\avgint_Qw^q\Big)^{1/q}
 &=\norm{w}_{L^{q,q}(Q)}
 \lesssim\norm{w}_{L^{q,s}(Q)},\\
 \Big(\avgint_Qw^{-p'}\Big)^{1/p'}
 &=\norm{w^{-1}}_{L^{p',p'}(Q)}
 \lesssim\norm{w^{-1}}_{L^{p',r'}(Q)}.
\end{align*}
Therefore
 $$
 [w]_{A_{p,q}}
 \lesssim [w]_{A_{(p,r),(q,s)}},
 $$
so $A_{(p,r),(q,s)}\subset A_{p,q}$. Together with (i), this proves (ii).
\end{proof}

Next, we consider the following properties of $A_{p,q}$ and $A_{(p,r),(q,s)}$.

\begin{proposition}\label{prop:OP}
Let $1<p\le q<\infty$ and $w\in A_{p,q}$ and set $\delta_0=1/p-1/q$. There exist pairs
$(p_-,q_-)$ and $(p_+,q_+)$ such that
\begin{equation*}
 1<p_-<p<p_+<\infty,
 \qquad
 1<q_-<q<q_+<\infty,
\end{equation*}
\begin{equation*}
 \frac1{p_-}-\frac1{q_-}
 =\frac1p-\frac1q
 =\frac1{p_+}-\frac1{q_+}=\delta_0,
\end{equation*}
and
\begin{equation*}
 w\in A_{p_-,q_-}\cap A_{p_+,q_+}.
\end{equation*}
\end{proposition}

\begin{proof}
Let $\varepsilon_0,\varepsilon_1>0$ be the reverse-H\"older exponents in
\eqref{eq:RH-wq}--\eqref{eq:RH-wp}.
Choose $p_-<p$ sufficiently close to $p$ that
\begin{equation}\label{eq:Pminus}
 p_-'<p'(1+\varepsilon_1),
\end{equation}
and define $q_-$ by
 $$
 \frac1{q_-}=\frac1{p_-}-\delta_0.
 $$
For $p_-$ close enough to $p$, we have $1<q_-<q$. Since $q_-<q$,
monotonicity of normalized Lebesgue norms gives
\begin{equation}\label{eq:Qminus}
 \Big(\avgint_Qw^{q_-}\Big)^{1/q_-}
 \le \Big(\avgint_Qw^q\Big)^{1/q}.
\end{equation}
By \eqref{eq:Pminus}, normalized Lebesgue monotonicity followed by
\eqref{eq:RH-wp} yields
\begin{align}
 \Big(\avgint_Qw^{-p_-'}\Big)^{1/p_-'}
 &\le
 \Big(\avgint_Qw^{-p'(1+\varepsilon_1)}\Big)^{1/[p'(1+\varepsilon_1)]}
\lesssim
 \Big(\avgint_Qw^{-p'}\Big)^{1/p'}.                  \label{eq:pminus-factor}
\end{align}
Multiplying \eqref{eq:Qminus} and \eqref{eq:pminus-factor} proves
$w\in A_{p_-,q_-}$.

Next choose $q_+>q$ sufficiently close to $q$ that
\begin{equation}\label{eq:qplus-choice}
 q_+<q(1+\varepsilon_0),
\end{equation}
and define $p_+$ by
 $$
 \frac1{p_+}=\delta_0+\frac1{q_+}.
 $$
Then $p_+>p$, and hence $p_+'<p'$. Therefore
\begin{equation}\label{eq:pplus-factor}
 \Big(\avgint_Qw^{-p_+'}\Big)^{1/p_+'}
 \le \Big(\avgint_Qw^{-p'}\Big)^{1/p'}.
\end{equation}
On the other hand, \eqref{eq:qplus-choice} and \eqref{eq:RH-wq} imply
\begin{align}
 \Big(\avgint_Qw^{q_+}\Big)^{1/q_+}
 &\le
 \Big(\avgint_Qw^{q(1+\varepsilon_0)}\Big)^{1/[q(1+\varepsilon_0)]}
 \lesssim
 \Big(\avgint_Qw^q\Big)^{1/q}.                      \label{eq:qplus-factor}
\end{align}
Multiplying \eqref{eq:pplus-factor} and \eqref{eq:qplus-factor} proves
$w\in A_{p_+,q_+}$.
\end{proof}

The preceding result gives the following consequence.

\begin{corollary}
Let $1<p\le q<\infty$ and $1\le r,s\le\infty$. Assume
$s\le q$, $r\ge p$, and $w\in A_{(p,r),(q,s)}$.  Then the conclusions of
\cref{prop:OP} hold.  Moreover, for every
$1\le\rho,\tau\le\infty$,
 $$
 w\in A_{(p_-,\rho),(q_-,\tau)}
 \cap A_{(p_+,\rho),(q_+,\tau)}.
 $$
\end{corollary}

\begin{proof}
The conclusion follows from \cref{thm:AA}, \cref{prop:OP} and
\cref{thm:AA}(i).
\end{proof}

\subsection{Power weights}

We next investigate that for which scope of real number $\gamma$,  $|x|^\gamma\in A_{(p,r), (q,s)}$ holds.

\begin{theorem}\label{thm:power-weights}
Let $1<p\le q<\infty$, $1\le r,s\le\infty$, and
$w_\gamma(x)=|x|^\gamma$ on $\R^d$. Then
$w_\gamma\in A_{(p,r),(q,s)}$ if and only if
\begin{align*}
 &\gamma>-\frac dq,
 \quad\text{or}\quad
 \gamma=-\frac dq\ \text{and}\ s=\infty;
 \\
 &\gamma<\frac d{p'},
 \quad\text{or}\quad
 \gamma=\frac d{p'}\ \text{and}\ r=1.
\end{align*}
\end{theorem}

\begin{proof}
We split the cubes into two classes. Write $Q=Q(x_Q,\ell(Q))$.

If $|x_Q|\ge4\sqrt d\,\ell(Q)$, then
 $$
 \frac12|x_Q|\le |x|\le\frac32|x_Q|,
 \qquad x\in Q.
 $$
Consequently,
\begin{align*}
 \norm{w_\gamma}_{L^{q,s}(Q)}&\approx |x_Q|^\gamma,
 &\norm{w_\gamma^{-1}}_{L^{p',r'}(Q)}&\approx |x_Q|^{-\gamma},
\end{align*}
and the product of the two norms is uniformly bounded.

If $|x_Q|<4\sqrt d\,\ell(Q)$, then
 $$
 Q\subset B(0,C_d\ell(Q)),
 \qquad |B(0,C_d\ell(Q))|\approx_d |Q|.
 $$
Thus it is enough to compute the norms on balls $B(0,R)$. If $\eta>-d/a$
\begin{equation}\label{eq:power-1}
 \norm{|x|^\eta\charac_{B(0,R)}}_{L^{a,b}}
 =R^{\eta+d/a}
 \norm{|x|^\eta\charac_{B(0,1)}}_{L^{a,b}}<\infty.
\end{equation}
At the critical exponent
$\eta=-d/a$, the decreasing rearrangement satisfies
 $$
 (|x|^{-d/a}\charac_{B(0,1)})^*(t)\approx t^{-1/a},
 \qquad 0<t<c_d,
 $$
so $b=\infty$ if and only if
\begin{equation}\label{eq:power-2}
 |x|^{-d/a}\charac_{B(0,1)}\in L^{a,b}.
\end{equation}
Apply \eqref{eq:power-1}--\eqref{eq:power-2} first to
$(\eta,a,b)=(\gamma,q,s)$ and then to
$(\eta,a,b)=(-\gamma,p',r')$. Therefore,
\begin{align*}
 \norm{w_\gamma}_{L^{q,s}(B(0,R))}\approx R^\gamma,\qquad
 \norm{w_\gamma^{-1}}_{L^{p',r'}(B(0,R))}\approx R^{-\gamma}.
\end{align*}
Base on the two cube cases above, we complete the sufficiency and necessity.
\end{proof}

From \cref{thm:power-weights}, it follows that the weak Lorentz classes are genuinely larger than $A_{p,q}$.

\begin{example}\label{ex:SE}
For every $1<p\le q<\infty$ and $1\le r\leq \infty$,
 $$
 |x|^{-d/q}\in A_{(p,r),(q,\infty)},
 $$
but $|x|^{-d/q}\notin A_{p,q}$.  By duality,
 $$
 |x|^{d/p'}\in A_{(p,1),(q,s)}\setminus A_{p,q}.
 $$
\end{example}

\section{Maximal operators and quantitative estimates}

Throughout this section,
\begin{equation}\label{eq:maximal-condition}
 0\le\alpha<d,\qquad 1<p<\infty,
 \qquad \frac1q=\frac1p-\frac\alpha d>0.
\end{equation}
Thus $q=p$ when $\alpha=0$, while $p<d/\alpha$ when $\alpha>0$.

\subsection{\texorpdfstring{$p>1$ and $r>s$}{p>1 and r>s}}

The inequality $r\le s$ is necessary by the following example.

\begin{example}\label{ex:r-greater-s}
Assume $1\le s<r\le\infty$.  Then
\begin{equation*}
 M_\alpha:L^{p,r}\longrightarrow L^{q,s}
\end{equation*}
does not hold.
\end{example}

\begin{proof}
Let $r<\infty$.  Choose
\begin{equation*}
 \frac1r<a\le\frac1s\qquad\text{and}\qquad
 f_a(x)=|x|^{-d/p}\bigl(\log(e/|x|)\bigr)^{-a}
 \mathbf1_{\{|x|<e^{-2}\}}(x).
\end{equation*}
Since
 $$
 f_a^*(t)\approx t^{-1/p}\bigl(\log(e/t)\bigr)^{-a},
 \qquad 0<t<c_d e^{-2d},
 $$
then $f_a\in L^{p,r}$ if and only if $ar>1$.
 A direct radial integration gives
\begin{align*}
 M_\alpha (f_a)(x)
 &\gtrsim |x|^{\alpha-d}
 \int_0^{2|x|}\rho^{d-1-d/p}
       \bigl(\log(e/\rho)\bigr)^{-a}\,d\rho \notag
\\
 &\gtrsim |x|^{-d/q}
       \bigl(\log(e/|x|)\bigr)^{-a}.
\end{align*}
By the choice of $a\in (\frac1r, \frac1s]$, we have $M_\alpha (f_a)\notin L^{q,s}$.

When $r=\infty$, take $a=0$. Then $f_0\in L^{p,\infty}$ and $M_\alpha (f_0)$ is not in any $L^{q,s}$ with
$s<\infty$.
\end{proof}

\subsection{\texorpdfstring{$p>1$ and $p\le r\le s\le q$}{p>1 and p<=r<=s<=q}}

We now consider $p\le r\le s\le q$, where the Lorentz--Muckenhoupt class agrees with $A_{p,q}$.

\begin{theorem}\label{thm:LOD}
Assume the hypotheses of \cref{thm:COF}. Let
$1<p\le q<\infty$ satisfy
 $$
 \frac1p-\frac1q=\frac1{p_0}-\frac1{q_0}.
 $$
Then, for every $w\in A_{p,q}$ and every $1\le r\le s\le\infty$,
\begin{equation*}
 \norm{Tf}_{L_w^{q,s}}
 \le C\norm{f}_{L_w^{p,r}}.
\end{equation*}
The constant depends on the indices, the function $\Psi$ in
\eqref{eq:base-weighted-bound}, and $[w]_{A_{p,q}}$.
\end{theorem}

\begin{proof}
Fix $w\in A_{p,q}$. By \cref{prop:OP}, there exist   pairs
$(p_-,q_-)$ and $(p_+,q_+)$  such that
 $$
 w\in A_{p_-,q_-}\cap A_{p_+,q_+}
 $$
with $1/p_{-}-1/q_{-}=1/p_{+}-1/q_{+}=1/p-1/q$. Then
\begin{align}
 \norm{Tf}_{L_w^{q_-}}\le C_-\norm{f}_{L_w^{p_-}}, \qquad
 \norm{Tf}_{L_w^{q_+}}\le C_+\norm{f}_{L_w^{p_+}}.
 \label{eq:LOD-1}
\end{align}
Choose $0<\theta<1$ so that
\begin{equation}\label{eq:LOD-2}
 \frac1p=\frac{1-\theta}{p_-}+\frac\theta{p_+}.
\end{equation}
Since
$1/q_\nu=1/p_\nu-(1/p-1/q)$ for
$\nu\in\{-,+\}$, the equality \eqref{eq:LOD-2} gives
\begin{equation*}
 \frac1q=\frac{1-\theta}{q_-}+\frac\theta{q_+}.
\end{equation*}
Real interpolation of the sublinear operator $T$,  together with \eqref{eq:LOD-1} and
\eqref{eq:WI}, yields for every $1\le t\le\infty$,
\begin{align*}
 \norm{Tf}_{L_w^{q,t}}
 &\le C_-^{1-\theta}C_+^\theta
 \norm{f}_{L_w^{p,t}}.
\end{align*}
Taking $t=r$ and using $L^{q,r}\hookrightarrow L^{q,s}$ for $r\le s$,
we obtain
 $
 \norm{Tf}_{L_w^{q,s}}
 \lesssim \norm{Tf}_{L_w^{q,r}}
 \lesssim \norm{f}_{L_w^{p,r}}.
 $
\end{proof}

\begin{theorem}\label{thm:SSM}
Assume \eqref{eq:maximal-condition} and
\begin{equation*}
 p\le r\le s\le q.
\end{equation*}
Then the following are equivalent:
\begin{enumerate}[label=\textup{(\roman*)}]
\item $w\in A_{(p,r),(q,s)}$;
\item $w\in A_{p,q}$;
\item $M_\alpha:L_w^{p,r}\to L_w^{q,s}$ is bounded.
\end{enumerate}
\end{theorem}

\begin{proof}
The equivalence of (i) and (ii) is \cref{thm:AA}.  The
implication (ii)$\Rightarrow$(iii) follows from \cref{thm:LOD}.  Conversely,
$A_{\alpha,Q}f\le M_\alpha f$ for every cube $Q$, therefore gives
$w\in A_{(p,r),(q,s)}$.
\end{proof}

\begin{remark}
The classes $A_{(p,p),(q,s)}$ with $s>q$ contain the critical powers in
\cref{ex:SE}, which are not $A_{p,q}$ weights.  Therefore an
estimate for every weight in $A_{(p,p),(q,s)}$ cannot follow
from classical $A_{p,q}$ extrapolation alone.
\end{remark}

The diagonal case gives a useful multiplier formulation of weighted Lorentz extrapolation.

\begin{corollary}
Let $1<p_0<\infty$ and suppose
 $$
 \norm{Tf}_{L^{p_0}(W)}
 \le \Psi([W]_{A_{p_0}})\norm{f}_{L^{p_0}(W)}
 $$
for every $W\in A_{p_0}$.  Then, for every $1<p<\infty$, every
$W\in A_p$, and $1\le r\le s\le\infty$,
\begin{equation*}
 \norm{Tf}_{L_{W^{1/p}}^{p,s}}
 \le C\norm{f}_{L_{W^{1/p}}^{p,r}}.
\end{equation*}
\end{corollary}

\begin{proof}
Apply \cref{thm:LOD} with $q_0=p_0$, $q=p$, and
$w=W^{1/p}$.  The condition $w\in A_{p,p}$ is equivalent to $W\in A_p$.
\end{proof}

The same argument also yields a vector-valued extension.

\begin{corollary}
Let $1<p_0\le q_0<\infty$. Suppose that, at the initial pair
$(p_0,q_0)$, one has for some $1<v<\infty$
 $$
 \Big\|\Big(\sum_j|Tf_j|^v\Big)^{1/v}\Big\|_{L_w^{q_0}}
 \le \Psi([w]_{A_{p_0,q_0}})
 \Big\|\Big(\sum_j|f_j|^v\Big)^{1/v}\Big\|_{L_w^{p_0}}.
 $$
Then for every $w\in A_{(p,r),(q,s)}$ with $p\le r\le s\le q$,
\begin{equation*}
 \Big\|\Big(\sum_j|Tf_j|^v\Big)^{1/v}\Big\|_{L_w^{q,s}}
 \le C\Big\|\Big(\sum_j|f_j|^v\Big)^{1/v}\Big\|_{L_w^{p,r}}.
\end{equation*}
\end{corollary}

\begin{proof}
Apply the scalar extrapolation theorem to the sublinear mapping
 $$
 \{f_j\}_j\longmapsto
 \Big(\sum_j|Tf_j|^v\Big)^{1/v}.
 $$
 The proof is completed.
\end{proof}

\subsection{\texorpdfstring{$p>1$, $r=p$ and $q\le s\le\infty$}{p>1, r=p, and q<=s<=infinity}}

The main new positive result treats the case $r=p$. For $q\le s\le\infty$, set
\begin{equation*}
 \mathcal A_{p,q}^{[s]}:=A_{(p,p),(q,s)}.
\end{equation*}
Then
\begin{equation*}
 A_{p,q}=\mathcal A_{p,q}^{[q]}
 \subset\mathcal A_{p,q}^{[s]}
 \subset\mathcal A_{p,q}^{[\infty]}.
\end{equation*}
Moreover,
\begin{equation*}
 [w]_{\mathcal A_{p,q}^{[\infty]}}
 =\sup_Q
 \Big(\frac1{|Q|}\|w^q\mathbf1_Q\|_{L^{1,\infty}}\Big)^{1/q}
 \Big(\frac1{|Q|}\int_Qw^{-p'}\Big)^{1/p'}.
\end{equation*}
Thus the endpoint is the multiplier class $A_{p,q}^{*}$; when $\alpha=0$
it is the class $A_p^*$ after writing the usual weight as $W=w^p$, see \cite{Sweeting2024}.

\begin{lemma}\label{lem:DLQ}
Let $0<q<\infty$, $q\le s\le\infty$, and let $\{h_j\}$ have pairwise
disjoint supports.  Then
\begin{equation*}
 \Big\|\sum_jh_j\Big\|_{L^{q,s}}^q
 \le C_{q,s}\sum_j\|h_j\|_{L^{q,s}}^q.
\end{equation*}
\end{lemma}

\begin{proof}
For $s=\infty$, disjointness gives
$d_{\sum h_j}(\lambda)=\sum_jd_{h_j}(\lambda)$, and hence
 $$
 \sup_{\lambda>0}\lambda^q d_{\sum h_j}(\lambda)
 \le\sum_j\sup_{\lambda>0}\lambda^q d_{h_j}(\lambda).
 $$
Let $s<\infty$ and $v=s/q\ge1$.  Using Minkowski's inequality in
$L^v((0,\infty),d\lambda/\lambda)$, we obtain
\begin{align*}
 \Big\|\sum_jh_j\Big\|_{L^{q,s}}^q
 &\approx
 \Big(\int_0^\infty
 \Big[\sum_j\lambda^q d_{h_j}(\lambda)\Big]^v
 \frac{d\lambda}{\lambda}\Big)^{1/v}\\
 &\le\sum_j
 \Big(\int_0^\infty
 [\lambda^q d_{h_j}(\lambda)]^v
 \frac{d\lambda}{\lambda}\Big)^{1/v}
 \approx\sum_j\|h_j\|_{L^{q,s}}^q.
\end{align*}
\end{proof}

\begin{lemma}\label{lem:DMI}
Let $w\in\mathcal A_{p,q}^{[s]}$ for some $q\le s\le\infty$, and write
$\sigma=w^{-p'}$.  There is a constant
$C=C(d,p,q,[w]_{\mathcal A_{p,q}^{[s]}})$ such that
\begin{equation*}
 \Big(\frac{|E|}{|Q|}\Big)^{2p'}
 \le C\frac{\sigma(E)}{\sigma(Q)}
\end{equation*}
for every cube $Q$ and every measurable $E\subset Q$.
Quantitatively one may take
\begin{equation*}
 \sigma(Q)\le C_{d,p,q}
 [w]_{\mathcal A_{p,q}^{[s]}}^{p'}\sigma(E)
\end{equation*}
whenever $|E|\ge c_d|Q|$.
\end{lemma}

\begin{proof}
Since $L^{q,s}\hookrightarrow L^{q,\infty}$,
$w\in\mathcal A_{p,q}^{[\infty]}$.  For $\alpha=0$, the result is the
reverse measure conclusion for the multiplier class $A_p^*$; for
$\alpha>0$ it is the corresponding conclusion for $A_{p,q}^*$.
Both follow from \cite[Lemmas~2.3 and 2.4]{Sweeting2024}.
\end{proof}

The two auxiliary lemmas lead to the main maximal operator theorem of the paper.

\begin{theorem}\label{thm:COC}
Assume \eqref{eq:maximal-condition} and let $q\le s\le\infty$.  Then
\begin{equation*}
 M_\alpha:L_w^{p,p}\longrightarrow L_w^{q,s}
 \quad\Longleftrightarrow\quad
 w\in\mathcal A_{p,q}^{[s]}.
\end{equation*}
Furthermore,
\begin{equation}\label{eq:COC-1}
 \|M_\alpha\|_{L_w^{p,p}\to L_w^{q,s}}
 \lesssim
 [w]_{\mathcal A_{p,q}^{[s]}}^{1+p'/q}.
\end{equation}
\end{theorem}

\begin{proof}
For a cube $Q$,
first apply the operator inequality to the truncations
$f_N=\min\{w^{-p'},N\}\mathbf1_Q$ and then let $N\to\infty$. The fact that
 $$
 M_\alpha f\ge |Q|^{\alpha/d-1}\int_Qw^{-p'},
 $$
and $1/p-1/q=\alpha/d$ yields
$[w]_{\mathcal A_{p,q}^{[s]}}\lesssim \|M_\alpha\|_{L_w^{p,p}\to L_w^{q,s}}$.

For sufficiency, it is enough to treat a dyadic
operator $M_\alpha^{\mathcal D}$.  Let $f\ge0$ be bounded and compactly
supported.  A standard level-set decomposition produces a sparse family
$\mathcal S\subset\mathcal D$ and pairwise disjoint sets $E_Q\subset Q$
with $|E_Q|\ge c_d|Q|$ such that
\begin{equation*}
 M_\alpha^{\mathcal D}f(x)
 \lesssim\sum_{Q\in\mathcal S}
 |Q|^{\alpha/d-1}\Big(\int_Qf\Big)\mathbf1_{E_Q}(x).
\end{equation*}
By \cref{lem:DLQ},
\begin{align*}
 \|wM_\alpha^{\mathcal D}f\|_{L^{q,s}}^q
 &\lesssim\sum_{Q\in\mathcal S}
 \Big[|Q|^{\alpha/d-1}\Big(\int_Qf\Big)
 \|w\mathbf1_{E_Q}\|_{L^{q,s}}\Big]^q \\
 &\le\sum_{Q\in\mathcal S}
 \Big[|Q|^{\alpha/d-1}\Big(\int_Qf\Big)
 \|w\mathbf1_Q\|_{L^{q,s}}\Big]^q.
\end{align*}
Write $\sigma=w^{-p'}$ and $g=fw^{p'}$.  Then
\begin{equation}\label{eq:COC-2}
 \int_Qf=\int_Qg\,d\sigma,
 \qquad
 \|g\|_{L^p(\sigma)}=\|wf\|_{L^p}.
\end{equation}
The definition of $A_{(p,p),(q,s)}$ gives
\begin{align*}
 &|Q|^{\alpha/d-1}\Big(\int_Qf\Big)
 \|w\mathbf1_Q\|_{L^{q,s}} \\
 &\qquad\le
 [w]_{\mathcal A_{p,q}^{[s]}}
 \langle g\rangle_{\sigma,Q}\,\sigma(Q)^{1/p}
 = [w]_{\mathcal A_{p,q}^{[s]}}
 \bigl[\sigma(Q)^{\alpha/d}\langle g\rangle_{\sigma,Q}\bigr]
 \sigma(Q)^{1/q}.
\end{align*}
Using \cref{lem:DMI} and the disjointness of
$E_Q$, we conclude that
\begin{align*}
 \|wM_\alpha^{\mathcal D}f\|_{L^{q,s}}^q
 &\lesssim [w]_{\mathcal A_{p,q}^{[s]}}^{q+p'}
 \sum_{Q\in\mathcal S}
 [\sigma(Q)^{\alpha/d}\langle g\rangle_{\sigma,Q}]^q
 \sigma(E_Q)\\
 &\le [w]_{\mathcal A_{p,q}^{[s]}}^{q+p'}
 \int_{\R^d}(M_{\alpha,\sigma}^{\mathcal D}g)^q\,d\sigma.
\end{align*}
Here
 $$
 M_{\alpha,\sigma}^{\mathcal D}g(x)
 :=\sup_{Q\in\mathcal D,\,Q\ni x}
 \sigma(Q)^{\alpha/d-1}\int_Q|g|\,d\sigma.
 $$
The dyadic fractional maximal inequality with respect to the measure
$\sigma$ gives
 $$
 \|M_{\alpha,\sigma}^{\mathcal D}g\|_{L^q(\sigma)}
 \lesssim_{p,q}\|g\|_{L^p(\sigma)}.
 $$
Combining this with \eqref{eq:COC-2} and taking the
$q$th root proves \eqref{eq:COC-1}.
\end{proof}

\begin{remark}
The proof of \cref{thm:COC} includes $\alpha=0$.  Hence, for
$1<p<\infty$ and $p\le s\le\infty$,
 $$
 M:L_w^{p,p}\to L_w^{p,s}
 \quad\Longleftrightarrow\quad
 w\in A_{(p,p),(p,s)}.
 $$
\end{remark}

\begin{remark}
Let $w_0(x)=|x|^{-d/q}$.  Then
 $$
 w_0\in\mathcal A_{p,q}^{[\infty]}\setminus A_{p,q},
 \qquad w_0^q(x)=|x|^{-d}\notin L^1_{\mathrm{loc}}.
 $$
Thus the endpoint case is not even a locally finite measure.
\end{remark}

\begin{proposition}\label{prop:QLO}
Suppose $q\le s\le\infty$ and let $\Gamma_s$ be an exponent for which
 $$
 \|M_\alpha\|_{L_w^{p,p}\to L_w^{q,s}}
 \lesssim [w]_{\mathcal A_{p,q}^{[s]}}^{\Gamma_s}
 $$
holds uniformly.  Then
\begin{equation}\label{eq:QLO}
 \Gamma_s\ge1+\frac{p'}s.
\end{equation}
\end{proposition}

\begin{proof}
For $0<\varepsilon<1$, set
 $$
 w_\varepsilon(x)=|x|^{(d-\varepsilon)/p'},
 \qquad
 f_\varepsilon=w_\varepsilon^{-p'}\mathbf1_{B(0,1)}.
 $$
Direct radial calculations give
 $$
 [w_\varepsilon]_{\mathcal A_{p,q}^{[s]}}
 \approx\varepsilon^{-1/p'},
 \qquad
 \|w_\varepsilon f_\varepsilon\|_{L^p}
 \approx\varepsilon^{-1/p}.
 $$
For $0<|x|<1/2$,
 $$
 w_\varepsilon(x)M_\alpha f_\varepsilon(x)
 \gtrsim \varepsilon^{-1}|x|^{-d/q+\varepsilon/p}.
 $$
Hence
 $$
 \frac{\|w_\varepsilon M_\alpha f_\varepsilon\|_{L^{q,s}}}
 {\|w_\varepsilon f_\varepsilon\|_{L^p}}
 \gtrsim\varepsilon^{-1/p'-1/s}
 \approx
 [w_\varepsilon]_{\mathcal A_{p,q}^{[s]}}^{1+p'/s},
 $$
which proves \eqref{eq:QLO}.
\end{proof}

For $r=p$ and $s=q$, write
\begin{equation*}
 \mathcal Q_{p,q}(w):=[w]_{A_{p,q}}^q.
\end{equation*}
The sharp estimates of Lacey, Moen, P\'erez, and Torres
\cite{LaceyMoenPerezTorres2010} are
\begin{align*}
 \|M_\alpha\|_{L_w^{p,p}\to L_w^{q,q}}
 &\lesssim \mathcal Q_{p,q}(w)^{(1-\alpha/d)p'/q},\\
 \|I_\alpha\|_{L_w^{p,p}\to L_w^{q,q}}
 &\lesssim \mathcal Q_{p,q}(w)^{(1-\alpha/d)\max\{1,p'/q\}},
\end{align*}
and both exponents are optimal.

For a locally finite positive measure $\sigma$ and $0<c<1$, define the reverse measure constant
\begin{equation*}
 [\sigma]_{\mathcal R_c}
 :=\sup_Q\sup_{\substack{E\subset Q\\ |E|\ge c|Q|}}
 \frac{\sigma(Q)}{\sigma(E)}.
\end{equation*}
The proof retains the following mixed quantitative information.

\begin{theorem}\label{thm:MQO}
Assume
 $$
 0\le\alpha<d,\qquad 1<p<\infty,
 \qquad \frac1q=\frac1p-\frac\alpha d>0,
 $$
and let $q\le s\le\infty$.  Write
$\sigma=w^{-p'}$ and
$\mathcal A_{p,q}^{[s]}=A_{(p,p),(q,s)}$.  There exists a dimensional
constant $c_d\in(0,1)$ such that
\begin{equation}\label{eq:MQO}
 [w]_{\mathcal A_{p,q}^{[s]}}
 \lesssim
 \|M_\alpha\|_{L_w^{p,p}\to L_w^{q,s}}
 \lesssim
 [w]_{\mathcal A_{p,q}^{[s]}}
 [\sigma]_{\mathcal R_{c_d}}^{1/q}.
\end{equation}
Moreover,
\begin{equation}\label{eq:MQO-1}
 [\sigma]_{\mathcal R_{c_d}}
 \lesssim
 [w]_{\mathcal A_{p,q}^{[s]}}^{p'},
\end{equation}
and consequently
 $$
 \|M_\alpha\|_{L_w^{p,p}\to L_w^{q,s}}
 \lesssim
 [w]_{\mathcal A_{p,q}^{[s]}}^{1+p'/q}.
 $$
\end{theorem}

\begin{proof}
The lower estimate in \eqref{eq:MQO} follows from
\cref{thm:OWT}.  In the proof of
\cref{thm:COC}, retain the factor needed to replace
$\sigma(Q)$ by $\sigma(E_Q)$ for the sparse sets $E_Q$.  Since
$|E_Q|\ge c_d|Q|$, this factor is at most
$[\sigma]_{\mathcal R_{c_d}}$.  Taking the $q$th root gives the upper
estimate in \eqref{eq:MQO}.  Finally,
\eqref{eq:MQO-1} is exactly the
quantitative conclusion of \cref{lem:DMI}.
\end{proof}

Combining \cref{thm:MQO} with
\cref{prop:QLO} yields
\begin{equation*}
 1+\frac{p'}s
 \le\Gamma_s\le1+\frac{p'}q,
 \qquad q\le s\le\infty,
\end{equation*}
where $1/\infty=0$.
for the optimal exponent in
 $$
 \|M_\alpha\|_{L_w^{p,p}\to L_w^{q,s}}
 \lesssim [w]_{\mathcal A_{p,q}^{[s]}}^{\Gamma_s}.
 $$
At $s=q$ the two exponents coincide and reproduce the classical sharp constant.
The gap for $s>q$ leads to the open question.
The endpoint $s=\infty$ is closely related to sharp multiplier weak-type
bounds.  Recent estimates of Lerner, Li, Ombrosi, and Rivera-R\'ios
\cite{LernerLiOmbrosiRivera2024a,LernerLiOmbrosiRivera2024b} show that this
question belongs to an active quantitative theory.

\subsection{\texorpdfstring{$p>1$, $r=1$, and $s<\infty$}{p>1, r=1, and s<infinity}}

The next example proves that $w\in A_{(p,r),(q,s)}$ is not sufficient for the case $r=1$.

\begin{example}\label{ex:WIF}
Let $1\le s<\infty$ and
\begin{equation*}
 w_1(x)=|x|^{d/p'}.
\end{equation*}
Then
$
 w_1\in A_{(p,1),(q,s)},
$
but
\begin{equation*}
 M_\alpha:L_{w_1}^{p,1}\not\longrightarrow L_{w_1}^{q,s}.
\end{equation*}
\end{example}

\begin{proof}
Let
 $
 f=w_1^{-1}\mathbf1_{B(0,1)}.
 $
Then $w_1f=\mathbf1_{B(0,1)}$, and hence $f\in L_{w_1}^{p,1}$.  If
$|x|>2$, a cube containing both $x$ and $B(0,1)$ gives
 $$
 M_\alpha f(x)
 \gtrsim |x|^{\alpha-d}
 \int_{B(0,1)}|y|^{-d/p'}\,dy
\approx |x|^{\alpha-d}.
 $$
Therefore
 $$
 w_1(x)M_\alpha f(x)
 \gtrsim |x|^{d/p'+\alpha-d}=|x|^{-d/q}.
 $$
Since $|x|^{-d/q}\notin L^{q,s}$ for any
finite $s$, then $M_\alpha$ is not bounded from $L_{w_1}^{p,1}$ to $L_{w_1}^{q,s}$.
\end{proof}

\subsection{\texorpdfstring{$p>1$: $r=1$, $s=\infty$; $1<r<p$, $r\le s\le\infty$; or $p<r\le s$, $s>q$}{p>1: r=1, s=infinity; 1<r<p; or p<r<=s, s>q}}

\begin{remark}\label{rem:PRS}
Let $p>1$, $r=1$ and $s=\infty$.  The class
$A_{(p,1),(q,\infty)}$ is necessary for
$M_\alpha:L_w^{p,1}\to L_w^{q,\infty}$.  On the other hand, \cref{thm:COC} with $L_w^{p,1}\hookrightarrow L_w^{p,p}$, shows that $A_{(p,p),(q,\infty)}$ is sufficient. However,
 $$
 A_{(p,p),(q,\infty)}\subsetneq A_{(p,1),(q,\infty)}
 $$
Whether the class $A_{(p,1),(q,\infty)}$ is sufficient remains open.
\end{remark}
\medskip

\begin{remark}\label{rem:II}
Let $p>1$, $1<r<p$, and $r\le s\le\infty$.  Boundedness of
$M_\alpha:L_w^{p,r}\to L_w^{q,s}$ implies
$w\in A_{(p,r),(q,s)}$ by \cref{cor:TWT}.  By
$L_w^{p,r}\hookrightarrow L_w^{p,p}$ and \cref{thm:COC}, the conditions $A_{(p,p),(q,s)}$ is sufficient. However, a complete characterization remains open.
\end{remark}
\medskip

\begin{remark}\label{rem:LILO}
For $p>1$, $p<r\le s$, and $s>q$,
$
 w\in A_{(p,r),(q,s)}
$
is a necessary condition for the boundedness of
$
 M_\alpha:L_w^{p,r}\longrightarrow L_w^{q,s}.
$
However, we have not established a corresponding sufficient condition
in this range.  In particular, it remains open whether
$A_{(p,r),(q,s)}$ is sufficient for the above boundedness.
\end{remark}

\subsection{\texorpdfstring{$p=1$}{p=1}}

In this subsection, we use the endpoint notation
\begin{equation*}
 [w]_{A_{(1,1),(q,\infty)}}
 :=\sup_Q
 \norm{w}_{L^{q,\infty}(Q)}
 \norm{w^{-1}}_{L^\infty(Q)}.
\end{equation*}

The case $s<\infty$ fails without a weight.

\begin{example}\label{prop:P1O}
Take $f=\mathbf1_{B(0,1)}$. If $|x|>2$, we have
$$M_\alpha f(x)\gtrsim |x|^{\alpha-d}=|x|^{-d/q}.$$
Then
 $$
 M_\alpha:L^{1,1}\longrightarrow L^{q,s}
 $$
does not hold for finite $s$.
\end{example}

\medskip

If $s=\infty$, the unweighted endpoint holds, but the multiplier weighted
setting is not yet obtained.

\begin{proposition}\label{prop:P1E}
Let
$
 0\le\alpha<d,
 q=\frac{d}{d-\alpha}.
$
If
$$
 M_\alpha:L_w^{1,1}(\R^d)
 \longrightarrow L_w^{q,\infty}(\R^d)
$$
is bounded, then
$
 w\in A_{(1,1),(q,\infty)}.
$
Conversely, the estimate holds whenever $w\in A_{1,q}$.
Thus
$$
 A_{1,q}
 \subset
 \bigl\{w:
 M_\alpha:L_w^{1,1}\to L_w^{q,\infty}
 \text{ is bounded}\bigr\}
 \subset
 A_{(1,1),(q,\infty)}.
$$
\end{proposition}

\begin{proof}
Necessity follows the same arguments as in \cref{cor:TWT}. The sufficiency for $w\in A_{1,q}$ is the classical multiplier
weak-type estimate for $M_\alpha$; see
\cite{CruzUribeSweeting2024}.  The sufficiency of the larger class
$A_{(1,1),(q,\infty)}$ is open.  When $\alpha=0$, this reduces to
the $p=1$ Muckenhoupt--Wheeden problem \cite{MuckenhouptWheeden1977}.
\end{proof}

\medskip

If $r>1$, the weak boundedness fails in the unweighted setting.

\begin{example}\label{ex:P1LI}
Let $p=1$, $r>1$, and
$
 q=\frac{d}{d-\alpha}.
$
Then
$$
 M_\alpha:L^{1,r}(\R^d)\longrightarrow L^{q,\infty}(\R^d)
$$
is not bounded.
\end{example}
\begin{proof}
Choose $a$ such that
$
 \frac1r<a\le1,
$
and let $f\ge0$ be supported in a fixed cube $Q_0$ and satisfy
$$
 f^*(t)
 =
 \frac{1}{t[\log(e/t)]^a},
 \qquad 0<t<e^{-2},
$$
and $f^*(t)=0$ for $t\ge e^{-2}$. Then
\begin{align*}
 \|f\|_{L^{1,r}}^r
 &=
 r\int_0^{e^{-2}}
 \bigl[t f^*(t)\bigr]^r\,\frac{dt}{t} =
 r\int_0^{e^{-2}}
 \frac{dt}{t[\log(e/t)]^{ar}}
 <\infty.
\end{align*}
On the other hand,
$$
 \|f\|_{L^1}
 =
 \int_0^{e^{-2}}f^*(t)\,dt
 =
 \int_0^{e^{-2}}
 \frac{dt}{t[\log(e/t)]^a}
 =\infty,
$$
since $a\le1$.

For $N\ge1$, set
$
 f_N:=\min\{f,N\}.
$
Since $0\le f_N\le f$, the property of Lorentz spaces gives
\[
 \sup_{N\ge1}\|f_N\|_{L^{1,r}}
 \le \|f\|_{L^{1,r}}<\infty.
\]
Moreover, by monotone convergence,
$
 \int_{Q_0}f_N(x)\,dx\longrightarrow\infty.
$

For every $x\in Q_0$, the cube $Q_0$ is admissible in the definition
of $M_\alpha$, and hence
$$
 M_\alpha f_N(x)
 \ge
 |Q_0|^{\alpha/d-1}\int_{Q_0}f_N.
$$
Therefore,
\begin{align*}
 \|M_\alpha f_N\|_{L^{q,\infty}}
 &\ge
 |Q_0|^{1/q}
 |Q_0|^{\alpha/d-1}
 \int_{Q_0}f_N =
 \int_{Q_0}f_N
 \longrightarrow\infty,
\end{align*}
where
$
 \frac1q=1-\frac{\alpha}{d}.
$
Thus $M_\alpha$ is not bounded from $L^{1,r}$ to $L^{q,\infty}$.
\end{proof}

\section{Applications}

\subsection{A two-weight commutator problem}

Throughout this section, let $0<\alpha<d$ and
\begin{equation*}
 1<p<\frac d\alpha,
 \qquad
 \frac1{p_\alpha}=\frac1p-\frac\alpha d.
\end{equation*}

\begin{lemma}\label{lem:weak-multiplier-identity}
Let $1<p<P<\infty$ and
 $$
 \frac1t=\frac1p-\frac1P.
 $$
Then,
\begin{equation}\label{eq:weak-multiplier-identity}
 \norm{F}_{L^{P,p}}^p
 \approx
 \sup_{\substack{v\ge0\\ \norm{v}_{L^{t,\infty}}\le1}}
 \norm{Fv}_{L^p}^p.
\end{equation}
The constants depend only on $p$ and $P$.
\end{lemma}

\begin{proof}
Since
 $$
 \left(\frac Pp\right)'=\frac tp,
 $$
the K\"othe dual of $L^{P/p,1}$ is $L^{t/p,\infty}$.  Then
\begin{align*}
 \norm{F}_{L^{P,p}}^p
 &=\norm{|F|^p}_{L^{P/p,1}}\approx
 \sup_{\substack{G\ge0\\ \norm{G}_{L^{t/p,\infty}}\le1}}
 \int_{\R^d}|F(x)|^pG(x)\,dx.
\end{align*}
Writing $G=v^p$ and using
 $$
 \norm{v^p}_{L^{t/p,\infty}}
 =\norm{v}_{L^{t,\infty}}^p,
 $$
we obtain
 $$
 \int_{\R^d}|F|^pG
 =\int_{\R^d}|Fv|^p
 =\norm{Fv}_{L^p}^p.
 $$
This proves \eqref{eq:weak-multiplier-identity}.
\end{proof}

We next establish the one-weight commutator characterization.

\begin{proposition}\label{prop:WTN}
Let $0<\alpha<d$, $1<p<d/\alpha$,
$1/p_\alpha=1/p-\alpha/d$, let $w\in A_{p,p_\alpha}$, and let
$1\le r,s\le\infty$.  If $Q$ and $Q'$ are
cubes of equal side length contained in a common cube $\widetilde Q$ with
$|\widetilde Q|\lesssim_d|Q|$, then
\begin{equation}\label{eq:weighted-test-product}
 \norm{w\charac_{Q'}}_{L^{p,r}}
 \norm{w^{-1}\charac_Q}_{L^{p_\alpha',s'}}
 \lesssim
 [w]_{A_{p,p_\alpha}}\,|Q|^{1+\alpha/d}.
\end{equation}
\end{proposition}

\begin{proof}
A standard consequence of $w\in A_{p,p_\alpha}$ is
 $$
 w^p\in A_p\subset A_\infty,
 \qquad
 w^{-p_\alpha'}\in A_{p_\alpha'}\subset A_\infty.
 $$
Applying \cref{lem:RHL} to $w$ and $w^{-1}$ gives
\begin{align*}
 \norm{w\charac_{Q'}}_{L^{p,r}}
 &\lesssim
 |Q|^{1/p}\Big(\avgint_{\widetilde Q}w^p\Big)^{1/p},
 \\
 \norm{w^{-1}\charac_Q}_{L^{p_\alpha',s'}}
 &\lesssim
 |Q|^{1/p_\alpha'}
 \Big(\avgint_{\widetilde Q}w^{-p_\alpha'}\Big)^{1/p_\alpha'}.
\end{align*}
Since $p<p_\alpha$, H\"{o}lder inequality yields
 $$
 \Big(\avgint_{\widetilde Q}w^p\Big)^{1/p}
 \le
 \Big(\avgint_{\widetilde Q}w^{p_\alpha}\Big)^{1/p_\alpha}
 $$
and
 $$
 \Big(\avgint_{\widetilde Q}w^{-p_\alpha'}\Big)^{1/p_\alpha'}
 \le
 \Big(\avgint_{\widetilde Q}w^{-p'}\Big)^{1/p'}.
 $$
Therefore
\begin{align*}
 &\norm{w\charac_{Q'}}_{L^{p,r}}
 \norm{w^{-1}\charac_Q}_{L^{p_\alpha',s'}}\lesssim
 |Q|^{1/p+1/p_\alpha'}
 \Big(\avgint_{\widetilde Q}w^{p_\alpha}\Big)^{1/p_\alpha}
 \Big(\avgint_{\widetilde Q}w^{-p'}\Big)^{1/p'}\le
 [w]_{A_{p,p_\alpha}}|Q|^{1+\alpha/d}.
\end{align*}
Therefore, the proof of \eqref{eq:weighted-test-product} is completed.
\end{proof}

We are now ready to give the one-weight Lorentz characterization of BMO.

\begin{theorem}\label{thm:one-weight-BMO-characterization}
Let $0<\alpha<d$, $1<p<d/\alpha$,
$1/p_\alpha=1/p-\alpha/d$, let $w\in A_{p,p_\alpha}$, and let
$1\le r\le s\le\infty$.  Suppose that $b$
is real-valued and locally integrable.  Then the following statements are
equivalent:
\begin{enumerate}[label=\textup{(\roman*)}]
\item $b\in\BMO$;
\item $
 [b,I_\alpha]$ is bounded from $L_w^{p,r}$ to $L_w^{p_\alpha,s}$.
\end{enumerate}
\end{theorem}

\begin{proof}
Assume first that $b\in\BMO$.  The weighted Lebesgue estimate for the fractional commutator is valid for
every $w\in A_{p,p_\alpha}$; see
\cite{Perez1995Commutators,SegoviaTorrea1991}.  Applying
\Cref{thm:LOD} gives
 $$
 \norm{[b,I_\alpha]f}_{L_w^{p_\alpha,s}}
 \lesssim
 \norm{b}_{\BMO}\norm{f}_{L_w^{p,r}},
 \qquad 1\le r\le s\le\infty.
 $$
This proves (i)$\Rightarrow$(ii).

Conversely, choose
$z_0\in\R^d$ so that
$z_0+[-1,1]^d$ does not meet the origin.  As in the classical necessity
argument, the function $|z|^{d-\alpha}$ has, on this cube, an absolutely
convergent Fourier expansion
\begin{equation}\label{eq:weighted-fourier-expansion}
 |z|^{d-\alpha}
 =\sum_{k\in\mathbb Z^d}a_ke^{i\nu_k\cdot z},
 \qquad
 \sum_{k\in\mathbb Z^d}|a_k|<\infty.
\end{equation}

Let $Q=x_0+\ell(-1/2,1/2)^d$ and set
 $$
 Q'=x_0-\ell z_0+\ell(-1/2,1/2)^d.
 $$
For $x\in Q$, define
$\varepsilon_Q(x)=\operatorname{sgn}(b(x)-b_{Q'})$.  The same computation
as in the unweighted Fourier-series argument gives
\begin{equation*}
 \int_Q|b-b_{Q'}|
 =\ell^{-\alpha}
 \sum_{k\in\mathbb Z^d}a_k
 \pairing{[b,I_\alpha]f_k}{g_k},
\end{equation*}
where
 $$
 f_k(y)=e^{-i\nu_k\cdot y/\ell}\charac_{Q'}(y),
 \qquad
 g_k(x)=e^{i\nu_k\cdot x/\ell}
 \varepsilon_Q(x)\charac_Q(x).
 $$
The associate space of $L_w^{p_\alpha,s}$ is
$L_{w^{-1}}^{p_\alpha',s'}$.  Therefore
\begin{align}
 \Big|\pairing{[b,I_\alpha]f_k}{g_k}\Big|
 &\lesssim
 \norm{f_k}_{L_w^{p,r}}
 \norm{g_k}_{L_{w^{-1}}^{p_\alpha',s'}} \notag\\
 &=\norm{w\charac_{Q'}}_{L^{p,r}}
 \norm{w^{-1}\charac_Q}_{L^{p_\alpha',s'}} \notag\\
 &\lesssim
 [w]_{A_{p,p_\alpha}}|Q|^{1+\alpha/d},
 \label{eq:WPB}
\end{align}
where \cref{prop:WTN} was used in the last estimate.  Since
$|Q|=\ell^d$,
 $$
 \ell^{-\alpha}|Q|^{1+\alpha/d}=|Q|.
 $$
Combining this identity with
\eqref{eq:weighted-fourier-expansion}--\eqref{eq:WPB}
yields
 $$
 \frac1{|Q|}\int_Q|b-b_{Q'}|
 \lesssim [w]_{A_{p,p_\alpha}}.
 $$
Finally,
 $$
 |b_Q-b_{Q'}|
 \le\frac1{|Q|}\int_Q|b-b_{Q'}|,
 $$
and hence
 $$
 \frac1{|Q|}\int_Q|b-b_Q|
 \le\frac2{|Q|}\int_Q|b-b_{Q'}|.
 $$
Taking the supremum over $Q$ proves (ii)$\Rightarrow$(i).
\end{proof}

\begin{remark}
Taking $r=s=p$ in \cref{thm:one-weight-BMO-characterization} gives
\begin{equation}\label{eq:sharp-r-equals-p}
 b\in\BMO
 \quad\Longleftrightarrow\quad
 [b,I_\alpha]:L^p(w^p)
 \longrightarrow L_w^{p_\alpha,p}.
\end{equation}
Because $p<p_\alpha$,
 $$
 L_w^{p_\alpha,p}\hookrightarrow
 L_w^{p_\alpha,p_\alpha}=L^{p_\alpha}(w^{p_\alpha}),
 $$
and the inclusion is strict for general functions.

Unweighted Lorentz characterizations and compactness results for fractional
commutators were obtained by Dao \cite{Dao2022} and by Hu and
Yan \cite{HuYan2026}.  Those results concern ordinary Lorentz spaces, whereas
\cref{thm:one-weight-BMO-characterization} treats $L_w^{p,r}$ for every $w\in A_{p,p_\alpha}$ and allows independent secondary
indices $1\le r\le s\le\infty$.  For $w\equiv1$, the space
$L^{p_\alpha,p}$ is the sharp Sobolev--Lorentz space associated with the Riesz
potential; see
\cite{ONeil1963,Alvino1977,GargSpector2014,SpectorVanSchaftingen2019}.
\end{remark}

Combining \cref{lem:weak-multiplier-identity} with \eqref{eq:sharp-r-equals-p} gives the partial answer.

\begin{theorem}\label{thm:family-to-BMO}
Let $0<\alpha<d$, $1<p<d/\alpha$, and define
 $$
 \frac1{p_\alpha}=\frac1p-\frac\alpha d,
 \qquad
 t=\frac d\alpha.
 $$
Let $w\in A_{p,p_\alpha}$, and let $b$ be real-valued and locally integrable.
Suppose that there exists $C_0>0$ such that
\begin{equation}\label{eq:uniform-weak-multiplier-family}
 \norm{[b,I_\alpha]f}_{L^p((wv)^p)}
 \le C_0\norm{v}_{L^{t,\infty}}
 \norm{f}_{L^p(w^p)}
\end{equation}
for every nonnegative
$v\in L^{t,\infty}$.  Then $b\in\BMO$, and
\begin{equation}\label{eq:family-BMO-bound}
 \norm{b}_{\BMO}
 \lesssim_{d,\alpha,p,[w]_{A_{p,p_\alpha}}} C_0.
\end{equation}
\end{theorem}

\begin{proof}
Since
 $$
 \frac1t=\frac\alpha d=\frac1p-\frac1{p_\alpha},
 $$
\cref{lem:weak-multiplier-identity} applies with $P=p_\alpha$.  For a bounded
compactly supported $f$, set
 $$
 F=[b,I_\alpha]f\,w.
 $$
The estimate \eqref{eq:uniform-weak-multiplier-family} gives
 $$
 \norm{Fv}_{L^p}
 \le C_0\norm{v}_{L^{t,\infty}}\norm{fw}_{L^p}
 $$
for every nonnegative $v\in L^{t,\infty}$. Using
\eqref{eq:weak-multiplier-identity}, we obtain
\begin{align*}
   \norm{[b,I_\alpha]f}_{L_w^{p_\alpha,p}}
 &=\norm{F}_{L^{p_\alpha,p}}\approx \sup_{\substack{v\ge0\\ \norm{v}_{L^{t,\infty}}\le1}}
 \norm{Fv}_{L^p}
  \lesssim C_0\norm{fw}_{L^p}
 =C_0\norm{f}_{L_w^{p,p}}.
\end{align*}
Applying  \eqref{eq:sharp-r-equals-p} proves
$b\in\BMO$ and \eqref{eq:family-BMO-bound}.
\end{proof}

\begin{remark}
The proof of \cref{thm:family-to-BMO} is based on Lorentz duality.  We need the condition that for some $w\in A_{p,p_\alpha}$ and all nonnegative $v\in L^{d/\alpha,\infty}$.
This differs from the approach in \cite{WangYin2025}, which is based on extrapolation and therefore requires estimates for all admissible choices
of both $w$ and $v$.
\end{remark}

\subsection{Weighted Hardy--Lorentz inequalities}

The classical Stein--Weiss theorem
\cite{SteinWeiss1958} is particularly well suited to the fixed-multiplier
convention.  The following Lorentz form is obtained by interpolation.

\begin{theorem}\label{thm:lorentz-stein-weiss}
Let $0<\alpha<d$, $1<p\le q<\infty$, and $\beta,\gamma\in\R$.  Assume
\begin{equation}\label{eq:stein-weiss-conditions}
 \beta<\frac dq,
 \qquad
 \gamma<\frac d{p'},
 \qquad
 \beta+\gamma\ge0,
 \qquad
 \frac1q=\frac1p-\frac{\alpha-\beta-\gamma}{d}.
\end{equation}
Then, for every $1\le r\le s\le\infty$,
\begin{equation}\label{eq:lorentz-stein-weiss}
 \norm{|x|^{-\beta}I_\alpha f}_{L^{q,s}}
 \le C\norm{|x|^{\gamma}f}_{L^{p,r}}.
\end{equation}
The constant depends only on
$d,\alpha,p,q,r,s,\beta$, and $\gamma$.
\end{theorem}

\begin{proof}
Set
 $$
 \delta:=\frac1p-\frac1q
 =\frac{\alpha-\beta-\gamma}{d}\ge0.
 $$
Choose $p_-<p<p_+$ sufficiently close to $p$ and define $q_-$ and $q_+$ by
 $$
 \frac1{p_-}-\frac1{q_-}
 =\frac1p-\frac1q
 =\frac1{p_+}-\frac1{q_+}=\delta.
 $$
Because all inequalities in \eqref{eq:stein-weiss-conditions} involving
$p$ and $q$ are strict, the pairs may be chosen so that
 $$
 \beta<\frac d{q_\pm},
 \qquad
 \gamma<\frac d{p_\pm'}.
 $$
Consequently, the
classical theorem gives
\begin{equation}\label{eq:stein-weiss- }
 \norm{|x|^{-\beta}I_\alpha f}_{L^{q_\pm}}
 \le C_\pm\norm{|x|^\gamma f}_{L^{p_\pm}}.
\end{equation}
Define
 $$
 Tg(x):=|x|^{-\beta}I_\alpha\big(|\cdot|^{-\gamma}g\big)(x).
 $$
Then \eqref{eq:stein-weiss- } says that
$T:L^{p_\pm}\to L^{q_\pm}$ is bounded.  Choose $0<\theta<1$ so that
 $$
 \frac1p=\frac{1-\theta}{p_-}+\frac\theta{p_+}.
 $$
The equality above implies
 $$
 \frac1q=\frac{1-\theta}{q_-}+\frac\theta{q_+}.
 $$
Real interpolation therefore yields
 $$
 \norm{Tg}_{L^{q,r}}\lesssim\norm{g}_{L^{p,r}},
 \qquad 1\le r\le\infty.
 $$
Taking $g=|\cdot|^\gamma f$ gives
 $$
 \norm{|x|^{-\beta}I_\alpha f}_{L^{q,r}}
 \lesssim\norm{|x|^\gamma f}_{L^{p,r}}.
 $$
Finally, $L^{q,r}\hookrightarrow L^{q,s}$ for $r\le s$,
\eqref{eq:lorentz-stein-weiss} is proved.
\end{proof}

Applying the Stein--Weiss estimate to Riesz potentials gives the corresponding Hardy--Sobolev inequality.

\begin{corollary}
Under the assumptions of \cref{thm:lorentz-stein-weiss}, for every
$1\le r\le s\le\infty$ and $u\in\mathcal S(\R^d)$,
\begin{equation*}
 \norm{|x|^{-\beta}u}_{L^{q,s}}
 \le C
 \norm{|x|^\gamma(-\Delta)^{\alpha/2}u}_{L^{p,r}}.
\end{equation*}
\end{corollary}

\begin{proof}
For $u\in\mathcal S(\R^d)$, the Fourier multiplier identity gives
 $$
 u=\kappa_{d,\alpha}I_\alpha\big(( -\Delta)^{\alpha/2}u\big)
 $$
with a dimensional normalization constant $\kappa_{d,\alpha}>0$.
Apply \cref{thm:lorentz-stein-weiss} to
$f=(-\Delta)^{\alpha/2}u$.
\end{proof}

The standard Hardy inequality
\begin{equation*}
 \Big\|\frac{u}{|x|}\Big\|_{L^p}
 \le \frac{p}{d-p}\|\nabla u\|_{L^p},
 \qquad 1<p<d,
\end{equation*}
fails at $p=d$.  The Lorentz refinement gives the following endpoint replacement.

\begin{theorem}\label{thm:critical-hardy-lorentz}
Let $d\ge2$ and $d\le q<\infty$.  Then, for every
$u\in C_c^\infty(\R^d)$,
\begin{equation}\label{eq:critical-hardy-lorentz}
 \big\||x|^{-d/q}u\big\|_{L^{q,\infty}}
 \le C_{d,q}\|\nabla u\|_{L^{d,1}}.
\end{equation}
For $q=d$ this is
\begin{equation*}
 \Big\|\frac{u}{|x|}\Big\|_{L^{d,\infty}}
 \le C_d\|\nabla u\|_{L^{d,1}}.
\end{equation*}
For every $r>1$, the right-hand side in
\eqref{eq:critical-hardy-lorentz} cannot be replaced uniformly by
$\|\nabla u\|_{L^{d,r}}$.
\end{theorem}

\begin{proof}
The pointwise Sobolev representation gives
\begin{equation}\label{eq:pointwise-sobolev-critical}
 |u(x)|\lesssim I_1(|\nabla u|)(x).
\end{equation}
O'Neil's endpoint convolution inequality yields
\begin{equation}\label{eq:oneil-endpoint-I1}
 \|I_1g\|_{L^\infty}\lesssim\|g\|_{L^{d,1}}.
\end{equation}
Since $|x|^{-d/q}\in L^{q,\infty}(\R^d)$, multiplication of
\eqref{eq:pointwise-sobolev-critical} by $|x|^{-d/q}$ and
\eqref{eq:oneil-endpoint-I1} prove \eqref{eq:critical-hardy-lorentz}.

Fix $r>1$ and let $R>e^2$. Choose a smooth radial function
$u_R\in C_c^\infty(\mathbb R^d)$ satisfying
\[
 u_R(x)=
 \begin{cases}
  \log R, & |x|\le R^{-1},\\
  \log(1/|x|), & 2R^{-1}\le |x|\le \frac12,\\
  0, & |x|\ge1,
 \end{cases}
\]
 and set
 $$
 v_R:=(\log R)^{-1/r}u_R.
 $$
On $2R^{-1}\le |x|\le \frac12$,
 $$
 |\nabla u_R(x)|=\frac1{|x|}.
 $$
Consequently, there exist constants $c_0,c_1,C_1>0$, independent of $R$,
such that
 $$
 c_1t^{-1/d}
 \le (|\nabla u_R|)^*(t)
 \le C_1t^{-1/d},
 \qquad
 c_0R^{-d}\le t\le c_0^{-1}.
 $$
Hence
\begin{align*}
 \|\nabla u_R\|_{L^{d,r}}^r
 &\approx
 \int_{c_0R^{-d}}^{c_0^{-1}}
 \bigl[t^{1/d}t^{-1/d}\bigr]^r\,\frac{dt}{t} =
 \int_{c_0R^{-d}}^{c_0^{-1}}\frac{dt}{t}
 \approx \log R.
\end{align*}
Therefore,
 $$
 \|\nabla v_R\|_{L^{d,r}}^r
 =(\log R)^{-1}\|\nabla u_R\|_{L^{d,r}}^r
 \approx 1,
 $$
and thus
 $$
 \|\nabla v_R\|_{L^{d,r}}\approx1.
 $$

On the ball $B_R:=B(0,R^{-1})$ one has
 $$
 v_R(x)=(\log R)^{1-1/r},
 $$
and therefore
 $$
 |x|^{-d/q}|v_R(x)|
 \ge
 R^{d/q}(\log R)^{1-1/r},
 \qquad x\in B_R.
 $$
By the definition of the weak Lorentz norm,
\begin{align*}
 \bigl\||x|^{-d/q}v_R\bigr\|_{L^{q,\infty}}
 &\ge
 R^{d/q}(\log R)^{1-1/r}|B_R|^{1/q}\\
 &=|B(0,1)|^{1/q}(\log R)^{1-1/r}
 \longrightarrow\infty
\end{align*}
as $R\to\infty$. Since
 $$
 \sup_{R>e^2}\|\nabla v_R\|_{L^{d,r}}<\infty,
 $$
there is no constant $C>0$ such that
 $$
 \bigl\||x|^{-d/q}u\bigr\|_{L^{q,\infty}}
 \le C\|\nabla u\|_{L^{d,r}}
 $$
holds for all $u\in C_c^\infty(\mathbb R^d)$.
\end{proof}

\begin{theorem}
Let $d\ge2$, $d\le q<\infty$, and let $\beta,\gamma>0$ satisfy $0<\gamma\le d/q$, $\gamma<d-1$ and
\begin{equation}\label{eq:critical-power-balance}
 \beta+\gamma=\frac dq.
\end{equation}
Then, for every $1\le r\le s\le\infty$,
\begin{equation*}
 \big\||x|^{-\beta}u\big\|_{L^{q,s}}
 \lesssim
 \big\||x|^\gamma\nabla u\big\|_{L^{d,r}}.
\end{equation*}
Moreover, the relation in \eqref{eq:critical-power-balance} is necessary.
\end{theorem}

\begin{proof}
Set
 $$
 \alpha=1,\qquad p=d,\qquad
 \beta+\gamma=\frac dq .
 $$
Then
 $$
 \frac1q=\frac1d-\frac{1-\beta-\gamma}{d},
 $$
which is precisely the condition in
\cref{thm:lorentz-stein-weiss}. Moreover,
 $$
 \beta=\frac dq-\gamma<\frac dq,
 \qquad
 \gamma<\frac d{d'}=d-1.
 $$
Hence \cref{thm:lorentz-stein-weiss} gives
 $$
 \bigl\||x|^{-\beta}u\bigr\|_{L^{q,s}}
 \lesssim
 \bigl\||x|^\gamma\nabla u\bigr\|_{L^{d,r}}.
 $$

It remains to prove the necessity of
\eqref{eq:critical-power-balance}. Suppose that
 $$
 \big\||x|^{-\beta}u\big\|_{L^{q,s}}
 \le C
 \big\||x|^\gamma\nabla u\big\|_{L^{d,r}}
 $$
holds for every  $u\in C_c^\infty(\mathbb R^d) $ . Define
 $$
 u_\lambda(x):=u(\lambda x),
 \qquad \lambda>0.
 $$
The dilation rule for Lorentz norms gives
 $$
 \big\||x|^{-\beta}u_\lambda\big\|_{L^{q,s}}
 =
 \lambda^{\beta-d/q}
 \big\||x|^{-\beta}u\big\|_{L^{q,s}},
 $$
whereas
 $$
 \big\||x|^\gamma\nabla u_\lambda\big\|_{L^{d,r}}
 =
 \lambda^{-\gamma}
 \big\||x|^\gamma\nabla u\big\|_{L^{d,r}}.
 $$
Therefore,
 $$
 \lambda^{\beta+\gamma-d/q}
 \big\||x|^{-\beta}u\big\|_{L^{q,s}}
 \le
 C\big\||x|^\gamma\nabla u\big\|_{L^{d,r}}
 \qquad\text{for every }\lambda>0.
 $$
Letting  $\lambda\to0 $  or  $\lambda\to\infty $  shows that
 $$
 \beta+\gamma-\frac dq=0.
 $$
Thus \eqref{eq:critical-power-balance} is necessary.
\end{proof}

\subsection{Fractional Schr\"odinger equations}

We next consider the fractional Schr\"odinger equation
\begin{equation}\label{eq:fractional-schrodinger-general}
 (-\Delta)^{\alpha/2}u-a(x)u=f
 \qquad\text{in }\mathcal S'(\R^d).
\end{equation}
A mild solution of \eqref{eq:fractional-schrodinger-general} is a function
satisfying
 $$
 u=\mathcal (-\Delta)^{-\alpha/2}(au+f).
 $$
Under \eqref{eq:stein-weiss-conditions}, assume further that
 $$
 \delta:=\alpha-\beta-\gamma>0,
 \qquad
 t:=\frac d\delta,
 \qquad
 A(x):=|x|^{\beta+\gamma}a(x),
 $$
and define
 $$
 X_\beta^{q,\rho}
 :=\{u:|x|^{-\beta}u\in L^{q,\rho}(\R^d)\}.
 $$

\begin{theorem}\label{thm:hardy-schrodinger}
Let $1\le\rho\le\infty$ and assume
$
 A\in L^{t,\infty}.
$
There exists $\varepsilon_0>0$ such
that
\begin{equation}\label{eq:small-bridge-potential}
 \norm{A}_{L^{t,\infty}}<\varepsilon_0,
\end{equation}
then \eqref{eq:fractional-schrodinger-general} has a unique mild
solution $u\in X_\beta^{q,\rho}$ for every
$|x|^\gamma f\in L^{p,\rho}$. Moreover,
\begin{equation}\label{eq:hardy-schrodinger-resolvent}
 \norm{|x|^{-\beta}u}_{L^{q,\rho}}
 \le
 \frac{C}{1-C\norm{A}_{L^{t,\infty}}}
 \norm{|x|^\gamma f}_{L^{p,\rho}}.
\end{equation}
\end{theorem}

\begin{proof}
Equip $X_{\beta}^{q,\rho}$ with the norm
 $$
 \norm{u}_{X_{\beta}^{q,\rho}}
 :=\norm{|x|^{-\beta}u}_{L^{q,\rho}}.
 $$
Since $q>1$, this Lorentz space has an equivalent Banach norm, including the
secondary endpoints.  Define
\begin{equation*}
 \mathcal T u:=\mathcal (-\Delta)^{-\alpha/2} f+\mathcal (-\Delta)^{-\alpha/2}(au).
\end{equation*}
The definitions of $\delta$ and $t$, together with
\eqref{eq:stein-weiss-conditions}, give
\begin{equation*}
 \frac1p=\frac1q+\frac1t.
\end{equation*}
It follows from
 $
 |x|^\gamma a(x)u(x)
 =A(x)|x|^{-\beta}u(x)
 $
that
\begin{equation*}
 \norm{|x|^\gamma au}_{L^{p,\rho}}
 \lesssim
 \norm{A}_{L^{t,\infty}}
 \norm{|x|^{-\beta}u}_{L^{q,\rho}}.
\end{equation*}
Applying \cref{thm:lorentz-stein-weiss} with $r=s=\rho$ gives
\begin{align}
 \norm{\mathcal T u}_{X_{\beta}^{q,\rho}}
 &\le C\norm{|x|^\gamma f}_{L^{p,\rho}}
   +C\norm{|x|^\gamma au}_{L^{p,\rho}} \notag\\
 &\le C\norm{|x|^\gamma f}_{L^{p,\rho}}
   +C\norm{A}_{L^{t,\infty}}
    \norm{u}_{X_{\beta}^{q,\rho}}.
 \label{eq:fixed-point-bound}
\end{align}
The same computation applied to $u-v$ gives
\begin{equation*}
 \norm{\mathcal T u-\mathcal T v}_{X_{\beta}^{q,\rho}}
 \le C\norm{A}_{L^{t,\infty}}
 \norm{u-v}_{X_{\beta}^{q,\rho}}.
\end{equation*}
Choose $\varepsilon_0$ so that $C\norm{A}_{L^{t,\infty}}\leq C\varepsilon_0<1$.  Banach's fixed-point theorem gives a unique $u$ satisfying
\begin{equation*}
 u=\mathcal (-\Delta)^{-\alpha/2} f+\mathcal (-\Delta)^{-\alpha/2}(au).
\end{equation*}
Applying $(-\Delta)^{\alpha/2}$ in $\mathcal S'$ proves
\eqref{eq:fractional-schrodinger-general}.  Moving the last term in
\eqref{eq:fixed-point-bound} to the left gives
\eqref{eq:hardy-schrodinger-resolvent}.
\end{proof}

The critical Hardy potential is obtained.

\begin{corollary}
Under the assumptions of \cref{thm:hardy-schrodinger}, there exists
$\lambda_0>0$ such that, for every $|\lambda|<\lambda_0$ and every
$f$ satisfying $|x|^\gamma f\in L^{p,\rho}$, the critical fractional
Hardy--Schr\"odinger equation
\begin{equation*}
 (-\Delta)^{\alpha/2}u
 -\lambda|x|^{-\alpha}u=f
 \qquad\text{in }\R^d
\end{equation*}
has a unique mild solution $u\in X_{\beta}^{q,\rho}$, and
\begin{equation*}
 \norm{|x|^{-\beta}u}_{L^{q,\rho}}
 \lesssim
 \norm{|x|^\gamma f}_{L^{p,\rho}}.
\end{equation*}
\end{corollary}

\begin{proof}
For $a(x)=\lambda|x|^{-\alpha}$,
 $$
 |x|^{\beta+\gamma}a(x)
 =\lambda|x|^{-(\alpha-\beta-\gamma)}
 =\lambda|x|^{-\delta}.
 $$
Since $t=d/\delta$, then
$
 \norm{|x|^{-\delta}}_{L^{t,\infty}(\R^d)}<\infty.
$
Thus \eqref{eq:small-bridge-potential} holds when $|\lambda|$ is
sufficiently small, this completes the proof.
\end{proof}

\vskip 0.2 true cm
{\bf \color{blue}{Conflict of Interest Statement:}}

\vskip 0.2 true cm

{\bf The authors declare that there is no conflict of interest in relation to this article.}

\vskip 0.2 true cm
{\bf \color{blue}{Data availability statement:}}

\vskip 0.2 true cm

{\bf  Data sharing is not applicable to this article as no data sets are generated
during the current study.}

\vskip 0.2 true cm

\end{document}